\documentclass[12pt,leqno]{article}
\usepackage{amsthm}
\usepackage{titlesec,amsfonts,amsmath}
\usepackage{amssymb,latexsym}
\usepackage{indentfirst}
\usepackage{color}
\usepackage[utf8]{inputenc}
\usepackage{mathrsfs}
\usepackage{bbm}
\usepackage{mathtools}
\usepackage[inline]{enumitem}
\usepackage{mathrsfs}

\usepackage{xcolor}
\usepackage{hyperref}
\hypersetup{
  colorlinks   = true,    
  urlcolor     = blue,    
  linkcolor    = blue,   
  citecolor    = red     
}

\newcommand{\curl}{\operatorname{curl}}

\newcommand{\dive}{\operatorname{div}}
\newcommand{\ep}{\varepsilon}
\newcommand{\zl}{{z_L}}
\newcommand{\ul}{{u_L}}
\newcommand{\wl}{{w_L}}

\newcommand{\hl}{{h_L}}
\newcommand{\Omegal}{{\Omega_L}}

\newcommand{\uep}{u_\varepsilon}
\newcommand{\uepn}{u_{\varepsilon_n}}
\newcommand{\unep}{u^n_\varepsilon}

\newcommand{\wep}{w_\varepsilon}
\newcommand{\wepn}{w_{\varepsilon_n}}
\newcommand{\zepn}{z_{\varepsilon_n}}
\newcommand{\wnep}{w^n_\varepsilon}

\newcommand{\jep}{J_\varepsilon}
\newcommand{\jepn}{J_{\varepsilon_n}}
\newcommand{\al}{\alpha}
\newcommand{\jal}{J_\alpha}

\newcommand{\zep}{z_\varepsilon}
\newcommand\nl[2]{\|#2\|_{L^{#1}}}
\def\dert{\partial_t }
\DeclareMathOperator*{\esssup}{ess\,sup}
 
\newcommand{\dd}{\,{\rm d}}

\newcommand\R{{\mathbb{R}}}
\renewcommand\P{{\mathbb{P}}}

\newcommand{\EE}{\mathbb{E}}

\renewcommand\div{{\rm div\,}}
\newcommand{\el}{\mathcal{E}_L}

\newtheorem{theorem}{Theorem}[section]
\newtheorem{propappendix}{Proposition}[section]
\newtheorem{proposition}[theorem]{Proposition}
\newtheorem{lemma}[theorem]{Lemma}
\newtheorem{corollary}[theorem]{Corollary}
\theoremstyle{definition}
\newtheorem{definition}[theorem]{Definition}

\theoremstyle{remark}
\newtheorem{remark}[theorem]{Remark}
\numberwithin{equation}{section}

\usepackage{framed,color}
\definecolor{shadecolor}{rgb}{0.9,0.9,0.9}

\renewcommand{\P}{\mathbb{P}}

\title{Asymptotic profiles and large-time behavior for 3D micropolar fluid equations with possibly vanishing spin viscosity}

\author{L. Brandolese, P. Braz e Silva, A. V. Busuioc, \\ D. Iftimie and C. F. Perusato}

\date\today

\begin{document}

\maketitle

\tableofcontents

\begin{abstract}
We consider 3D micropolar flows with possible vanishing spin viscosity and investigate the decay of the energy for large times. We compute first the exact $L^2$-asymptotic profile, as $t\to+\infty$, for solutions to the linear 3D micropolar equations, up to the second order. For the nonlinear micropolar system, we first establish the existence of restricted Leray solutions. This new notion of solutions is required because it is not known whether the weak finite energy solutions verify a strong energy inequality. Next, we study the large-time behavior of restricted Leray solutions, and prove that they behave asymptotically in $L^2$ like their linear counterpart, up to the critical algebraic decay rate $O(t^{-5/2})$ for the energy. Applying a remarkable linear enstrophy identity, we show that the microrotation field exhibits faster decay in $L^2$ than the velocity field, allowing us to impose our hypothesis on the velocity field only and not on the angular velocity.



\end{abstract}

\section{Introduction}

In this paper, we investigate the large-time behavior of solutions to the three-dimensional micropolar fluid equations. We consider the nonlinear system, including the case of vanishing spin viscosity, which is physically relevant, as discussed in \cite{brandolese_2d_2024}. This situation is particularly challenging, as it leads to a loss of regularizing effects in the microrotation equation. As we shall see below, the 3D case considered here is quite more complicated than the 2D case studied in \cite{brandolese_2d_2024} for at least two reasons: the lack of an enstrophy identity for the nonlinear system (see \cite[Section 3]{brandolese_2d_2024}) and the lack of a strong energy inequality (see \cite[Proposition 2.5]{brandolese_2d_2024}). We will discuss this in detail later.

The micropolar fluids equations model a class of fluids where the microstructure is taken into account. Physically, they may represent fluids where the rotation of suspended micro-particles is considered, and can be used to model polymeric suspensions, liquid crystals, and blood flow, for example. They were introduced by Eringen in \cite{Eringen}, and are a generalization of the classical Navier-Stokes equations. Details on the derivation of the model may also be seen in \cite{Condiff-Dahler} and \cite{Lukaszewicz}. In \cite{Lukaszewicz}, there is also a detailed mathematical analysis of the model, including many existence and uniqueness results for both stationary and non-stationary problems for such models.  

The system describing the motion of micropolar fluids in the whole space $\R^3$ is  
\begin{equation}
\label{MP}
\left\{
\begin{aligned}
&\dert u+u\cdot \nabla u+\nabla p=(\mu+\chi)\Delta u+2\chi\curl  w, \qquad 
x\in\R^3, \;t>0 , \\
&\dert w+u\cdot\nabla w=\gamma\Delta w+\kappa\nabla(\dive  w)+2\chi\curl  u-4\chi w, \\
&\dive  u=0,  
\end{aligned}
\right.
\end{equation}
 where $u$ is the translational velocity and $w$ is the angular velocity of the micro-particles in the fluid. The coefficients $\mu, \chi > 0$ denote the kinematic and vortex viscosities, respectively, while $\kappa,\gamma \geq 0$ denote the gyroviscosity and the spin viscosity, respectively. These viscosities are assumed to be constant throughout this work.
The initial data associated with system \eqref{MP} are denoted, as usual, by $u_0$ and $w_0$.
 
The three-dimensional problem differs in an essential way from the planar case. In 3D, the microrotation field no longer has a fixed direction, so the planar enstrophy identity used in the 2D analysis \cite{brandolese_2d_2024} (which is crucial in the spin-viscosity-free regime $\gamma = 0$) is no longer available. The lack of an enstrophy-like identity prevents the use of the monotonicity argument \cite{GuterresNichePerusatoZingano2023, ZinganoJMFM} which was quite useful in the two-dimensional setting, as shown in \cite{brandolese_2d_2024}. The analysis here requires a new notion of solution to deal with the lack of a strong energy inequality and combines a refined Fourier splitting argument with the damping mechanism in the $w$-equation which compensates for the absence of diffusion when $\gamma =0$.

In our first result, we obtain the precise behavior of the \emph{linear} micropolar system
\begin{equation}\label{LMP1}
\left\{
\begin{aligned}
&\dert \ul =(\mu+\chi)\Delta \ul+2\chi \curl  \wl, \\
&\dert \wl =\gamma\Delta \wl+\kappa\nabla(\dive   \wl)+2\chi\curl  \ul-4\chi \wl,\\
&\dive \ul=0.
\end{aligned}
\right.
\qquad 
x\in\R^3, \;t>0
\end{equation}
We observe that the absence of the pressure in the equation of $\ul$ is due to the fact that all the terms in that equation are divergence free. We will prove in Section \ref{linearsect} the following result. 
\begin{theorem}
\label{th:linear-pro}
Assume $\gamma\ge0$.
The solution $(\ul,\wl)$ of the linear problem~\eqref{LMP1} with initial data $(u_0,w_0)\in L_\sigma^2(\R^3)\times L^2(\R^3)$ has the following asymptotic behavior in $L^2$ as $t\to+\infty$:
\begin{equation*}
\big\|\ul-e^{\mu t\Delta}u_0-\frac12\curl e^{\mu t\Delta}w_0\big\|_{L^2}\leq \frac Ct \|u_0\|_{L^2}+\frac C{t^{\frac32}} \|\P w_0\|_{L^2},
\qquad\forall t\geq1
\end{equation*}
and 
\begin{equation*}
\big\|\wl-\frac12\curl e^{\mu t\Delta}u_0+\frac14\Delta e^{\mu t\Delta}\P w_0\big\|_{L^2}\leq\frac C{t^{\frac32}} \|u_0\|_{L^2}+\frac C{t^2} \|w_0\|_{L^2},
\qquad\forall t\geq1
\end{equation*}
for some constant $C$ depending only on the material coefficients.
\end{theorem}

Our next result deals with the full nonlinear system~\eqref{MP} when the initial data $u_0$ and $w_0$ are square integrable. Let us comment first on the existence of the solutions of \eqref{MP} under this assumption.  The standard energy estimate for \eqref{MP}, see relation \eqref{enerzero}, together with a compactness argument implies the existence of the so-called Leray solutions of \eqref{MP}. These are finite energy solutions satisfying \eqref{MP} in the sense of the distributions (one needs to write the nonlinear term  $u\cdot \nabla w$ as $\dive  (u\otimes w)$) and such that the weak energy inequality \eqref{enerzero} holds true. Unfortunately, the weak energy inequality \eqref{enerzero} does not suffice to obtain the decay of the energy at large times. We need instead the so-called strong energy inequality \eqref{strongen}. When $\gamma>0$, the existence of global Leray solutions
satisfying the strong energy inequality~\eqref{strongen} below
can be established exactly as for the Navier--Stokes equations. But in the case $\gamma=0$, which is the focus of the present work, it is not clear why solutions of \eqref{MP} verifying the strong energy inequality \eqref{strongen} should exist. To deal with this issue, we will work with a class of solutions which we call \emph{restricted Leray solutions} following a terminology introduced by P.G.~Lemari\'e-Rieusset~\cite{Lem02}. The restricted Leray solutions are solutions obtained as a limit of solutions of some mollified approximation of \eqref{MP}, see Definition \ref{def:restricted} and Proposition \ref{prop:exist} below. 
 
For restricted Leray solutions, we are able to obtain
sharp $L^2$-decay results for large time.
The next theorem improves several partial results available in the literature in this direction.
Its statement is inspired by Wiegner's classical decay result \cite{Wie87} for the Navier-Stokes equation.
It allows one to obtain solutions to the micropolar system with algebraic decay of the energy
up to the rate $O((1+t)^{-5/2})$.
As it will be seen, our analysis takes fully advantage of the special structure of the micropolar system, thus leading to a deeper understanding on how microrotation affects the large time behavior of the fluid motion.

\begin{theorem}
\label{th:decay}
Let $\mu,\chi>0$ and $\gamma,\kappa\ge0$.
Let $z:=(u,w)$ be a restricted Leray solution
to~\eqref{MP}, with initial data $z_0:=(u_0,w_0)\in L^2_\sigma(\R^3)\times L^2(\R^3)$.
The following hold true:
\begin{itemize}
\item[i)] $\|u(t)\|_{L^2}^2+\|w(t)\|_{L^2}^2\to0$ as $t\to+\infty$.
\item[ii)] If $0\le \Gamma\le 5/2$ and the solution $(\ul,\wl)$ of the linear problem~\eqref{LMP1}
with initial data $(u_0,w_0)$ satisfies $\|\ul(t)\|_{L^2}^2=O(t^{-\Gamma})$,
 then $\|u(t)\|_{L^2}^2+\|w(t)\|_{L^2}^2=O(t^{-\Gamma})$.
\item[iii)] Under the condition of the previous item (in the case $\gamma=0$ we also require the additional
condition $\int (1+|\xi|)|\widehat z_0(\xi)|\dd\xi <\infty $), one has
	\begin{equation}\label{fluctuation}
	\|(u-\ul)(t)\|_{L^2}^2+\|(w-\wl)(t)\|_{L^2}^2\lesssim
	\begin{cases}
	(1+t)^{-1/2}\zeta(t) , &\text{if $\Gamma=0$},\\
	(1+t)^{-1/2-2\Gamma} , & \text{if $0<\Gamma<1$} , \\
	(1+t)^{-5/2}(\log(e+t))^2, &\text{if $\Gamma=1$} , \\
	(1+t)^{-5/2} , &\text{if $1<\Gamma\le 5/2$} , 
	\end{cases}
	\end{equation}
	with $\lim_{t\to+\infty}\zeta(t)=0$.
\end{itemize} 
\end{theorem}
We established in Theorem \ref{th:linear-pro} some approximate formulas for $\ul$ and $\wl$ up to a $O(\frac1t)$ error for $\ul$ and up to a  $O(\frac1{t^{3/2}})$ error for $\wl$. So we can replace in \eqref{fluctuation} the term $\wl$ with the quantity $\frac12\curl e^{\mu t\Delta}u_0-\frac14\Delta e^{\mu t\Delta}\P w_0$ for all $\Gamma$, and we can replace $\ul$ by $e^{\mu t\Delta}u_0+\frac12\curl e^{\mu t\Delta}w_0$ for any $\Gamma\leq\frac34$.

Results on the decay of the energy for the micropolar system were established before in~\cite{BCFZ,Guterres2,Niche-Perusato}, but Theorem~\ref{th:decay} improves them
in several ways. In particular, we note that 
Theorem~\ref{th:decay} remains valid in the case $\gamma=0$,
whereas the condition $\gamma>0$ was required, and used in an essential way, in~\cite{BCFZ,Guterres2,Niche-Perusato}.
Including the partially inviscid case $\gamma=0$ represents an important new technical contribution of the present work.

As in classical Wiegner's theorem~\cite{Wie87}, the $L^2$-decay of the solutions agrees
with the decay of the linear part up to a critical decay rate, beyond which
the nonlinear effects are predominant. 
The restriction $\Gamma\le 5/2$ also appears in the case of the
Navier-Stokes equations and it is known to be optimal.
Remarkably, in assertions~\emph{ii)-iii)}, no assumptions
on the decay of the linear microrotation $\|\wl(t)\|_{L^2}^2$ is needed: this is due
to the surprising fact that the algebraic decay of $\|\ul(t)\|_{L^2}^2$ forces
a faster algebraic decay for $\|\wl(t)\|_{L^2}^2$. For this, see  Proposition \ref{prop:impro} in the Appendix.

Our method is somehow close to that of~\cite{BCFZ}, based on the Fourier splitting method introduced in~\cite{Sch85}.
The novelty of our approach relies in taking advantage of the presence of a damping term in the equation for~$w$ and showing that this term compensates the lack of dissipation in this equation.
Results for $\gamma=0$ were addressed, for the 2D case, in~\cite{Guo-Jia-Dong}  at a formal level (i.e., relying on regularity properties of solutions that are not rigorously established) and, in a much more systematic way, in~\cite{brandolese_2d_2024}. For the three-dimensional case, in a recent paper \cite{Niu-Shang}, the authors obtained corresponding decay estimates, when $\gamma=0$, under much stronger assumptions (both smoothness and smallness) on the initial data. More precisely, assuming $(u_0,w_0) \in (L^1  \cap H^s_\sigma) \times (L^1  \cap H^s)$, with $s>3/2$ and sufficiently small in $H^s$, they proved that global strong solutions exhibit heat-like decay rates for $u$, together with the improved decay for the microrotation $w$, which, in particular, corresponds to the case $\Gamma = 3/2$ in our result (see also Corollary \ref{cor:decay} below). It is also worth mentioning that the 2D case has a quite different structure, as the micro-rotational angular velocity $w$ has a fixed direction, 
and crucial features of the planar case are no longer valid in~3D. For example, the 2D enstrophy identity shown in ~\cite{brandolese_2d_2024}, which has deep consequences
on the large time behavior. Another important contribution of the present work is deriving and applying to the decay problem 
a new enstrophy-type identity for the 3D \emph{linearized problem}. 

\begin{remark}
Let $z_0=(u_0,w_0) \in L^2_\sigma(\R^3)\times L^2(\R^3)$
and let us compare, for example, our result with the ones  in~\cite{BCFZ,Niche-Perusato}.
In~\cite{BCFZ}, the authors dealt with the case $\gamma>0$ and 
proved that, if one additionally assumes
 $z_0\in L^1(\R^3)\times L^1(\R^3)$,
 then $\|u(t)\|_{L^2}^2+\|w(t)\|_{L^2}^2=O(t^{-3/2})$ as $t\to+\infty$.
Our theorem, besides encompassing the case $\gamma=0$, allows one to obtain the
same decay rate under the more general condition
$\|\ul(t)\|_{L^2}^2=O(t^{-3/2})$:
the result of Theorem~\ref{th:linear-pro} implies that
for the latter condition it is enough to assume that $u_0+\frac12\curl w_0\in \dot B^{-3/2}_{2,\infty}(\R^3)$,
since $f\in \dot B^{-3/2}_{2,\infty}(\R^3)$ if and only if
$\|e^{\mu t\Delta}f\|_{L^2}^2=O(t^{-3/2})$, see~\cite[Chapter 2]{BahCD11}.
This technical improvement is important because it shows
that solutions of the micropolar system \emph{can decay faster} than
the solution of the Navier-Stokes equation with the same
initial velocity~$u_0$, since the latter condition
holds without the need of $u_0\in  \dot B^{-3/2}_{2,\infty}(\R^3)$.
Similarly, Theorem~\ref{th:decay} may be used to improve the result in~\cite{Niche-Perusato}: therein, optimal decay rates of solutions (for $\gamma>0$) are obtained under assumptions on the \emph{decay character} of $u_0$ and $w_0$.
But, in fact, a more general assumption on the decay character of $u_0-\frac12\curl w_0$
would be enough.  
\end{remark}

We finish this introduction by discussing improved decay rates for $\|w(t)\|_{L^2}^2$. 
This issue is physically relevant due to the strong alignment phenomena of $w$ with the vorticity discovered
in~\cite{Guterres}. 
Results in this direction appeared 
in~\cite{Guterres2}, \cite{Guterres3}, but only for $\gamma>0$.
In all these references, the dissipation term $\gamma\Delta w$, in the equation for $w$, plays a crucial role in the arguments.
Corollary~\ref{cor:decay} below illustrates that an improved decay rate for $\|w(t)\|_{L^2}^2$ can also be obtained when $\gamma=0$.
Here, we just state a simple particular case of this corollary:

\emph{If $u_0\in (L^1\cap L^2_\sigma)(\R^3)$ and $w_0\in (L^{3/2}\cap L^2)(\R^3)$ (if $\gamma=0$, we also require $\int (1+|\xi|)|\widehat z_0(\xi)|\dd\xi <\infty $), then, as $t\to+\infty$,}
\[
\begin{split}
&\|u(t)\|_{L^2}^2=O(t^{-3/2})
\quad\text{(same decay rate as for the heat equation}) , \\
&\|w(t)\|_{L^2}^2=O(t^{-5/2})
\quad\text{(the decay for the heat equation would be slower:
$\|e^{\mu t\Delta }w_0\|_{L^2}^2=O(t^{-1/2})$}).
\end{split}
\]

\subsection*{Notations}
The symbol $\mathbb{P}$ denotes the $L^2$ orthogonal projector onto divergence-free vector fields, and $L^2_\sigma(\R^3)$ will represent the space of 3D vector fields in $L^2(\R^3)$ that are divergence-free in the sense of distributions. 
In our estimates, we will often write 
$f(t)\lesssim g(t)$ to indicate there is a constant $C>0$, independent on time, such that $f(t)\le C g(t)$. The symbol $\mathcal{F}$ denotes the Fourier transform. We will also use the classical notation $H^s$ for $L^2$ based Sobolev spaces and $\dot B^{s}_{p,q}$ for the homogeneous Besov spaces.

\section{Restricted Leray solutions for the micropolar system}

The aim of this section is to prove the existence of a special class of Leray solutions of the micropolar fluids equations \eqref{MP}. The $L^2$ energy estimates for \eqref{MP} consists in multiplying the first equation by $u$, the second equation by $w$, integrating in space and adding. Doing some integrations by parts and using the classical cancellations $\int u\cdot\nabla u \cdot u=\int u\cdot\nabla w \cdot w=\int\nabla p\cdot u=0$ we formally obtain the following relation:
\begin{equation*}
\frac12\partial_t (\nl2u^2+\nl2w^2)+(\mu+\chi)\nl2{\nabla u}^2+\gamma\nl2{\nabla w}^2+\kappa\nl2{\div w}^2+4\chi\nl2w^2-4\chi \int_{\R^3}w\cdot(\curl  u)=0
\end{equation*}
Observing, by integration by parts, that
\begin{equation*}
\nl2{\nabla u}=\nl2{\curl  u},
\end{equation*}
the energy estimate can be written under the form 
\begin{equation}\label{energ}
\frac12\partial_t (\nl2u^2+\nl2w^2)+\mu\nl2{\nabla u}^2+\gamma\nl2{\nabla w}^2+\kappa\nl2{\div w}^2+\chi\nl2{\curl  u-2w}^2=0.
\end{equation}
The above relation is formal, it must be viewed as an a priori estimate. 

The existence of finite energy solutions of \eqref{MP} in the case $\gamma>0$ follows in the same manner as for the classical incompressible Navier-Stokes equations. The case $\gamma=0$ is a bit more delicate, but still easily doable with classical compactness methods (see \cite{brandolese_2d_2024}). Indeed, if $\uep,\wep$ is a sequence of suitable approximate solutions, the $L^2$ estimates from \eqref{energ} imply $L^2_tH^1_x$ estimates for $u_\ep$, which in turn implies compactness in $L^2_{loc}$ for  $u_\ep$. For $\wep$ we have no $H^1$ estimates in $x$, so we get only weak convergence in $L^2$. But one can easily check that strong convergence for $\uep$ in $L^2_{loc}$ and weak convergence for $\wep$ in $L^2$ suffices to pass to the limit in \eqref{MP}. Taking the $\liminf\limits_{\ep\to0}$ in the energy estimate for $(\uep,\wep)$ and using the weak lower semicontinuity of norms, one can also easily prove that the solution obtained with such a limiting process verifies the energy inequality
\begin{multline}
\label{enerzero}
\nl2{u(t)}^2+\nl2{w(t)}^2+2\mu\int_0^t \nl2{\nabla u}^2+2\gamma\int_0^t\nl2{\nabla w}^2+2\kappa\int_0^t\nl2{\div w}^2\\
+2\chi\int_0^t\nl2{\curl  u-2w}^2
\leq\nl2{u_0}^2+\nl2{w_0}^2
\end{multline}
for all $t\geq0$.

Unfortunately, the method of proof that we will employ in Section~\ref{sec:decay} requires the following strong energy inequality to hold true:
\begin{multline}\label{strongen}
\nl2{u(t)}^2+\nl2{w(t)}^2+2\mu\int_s^t \nl2{\nabla u}^2+2\gamma\int_s^t\nl2{\nabla w}^2+2\kappa\int_s^t\nl2{\div w}^2\\
+2\chi\int_s^t\nl2{\curl  u-2w}^2
\leq\nl2{u(s)}^2+\nl2{w(s)}^2
\end{multline}
for all $t\geq s$ and for almost all $s\geq0$. Such an inequality can be proved in the case $\gamma>0$ with the same proof as for the Navier-Stokes equations. In the case $\gamma=0$, the proof given in the case of the Navier-Stokes equations does not go through to the micropolar fluid equations. Let us briefly explain why. The method of proof of the  strong energy inequality consists in writing \eqref{strongen} for 
solutions $(\uep,\wep)$, of a mollified system, in taking  the $\liminf\limits_{\ep\to0}$  and using the weak lower semicontinuity of norms to deal with the left-hand side. To pass to the limit in the right-hand side we would need strong convergence in $L^2$ for $\uep(s)$ and $\wep(s)$. This is essentially true for $\uep(s)$ for almost all $s$ due to the strong convergence of $\uep$ in $L^2_{loc}(\R_+\times\R^3)$. But for the $\wep$ we only have weak $L^2$ convergence and this is why the method fails. The existence of weak solutions of \eqref{MP} verifying the strong energy inequality \eqref{strongen} is open in the case $\gamma=0$.

To deal with this issue, we will use a notion of a so-called restricted Leray solution. This terminology is borrowed from the book of Lemarié-Rieusset~\cite{Lem02}. The construction of restricted Leray solutions proceeds as follows.

Let $\theta\in C^\infty_0(\R^3)$ be a non-negative and even function, such that $\int\theta=1$ and let, for 
$\varepsilon>0$ and $x\in\R^3$, $\theta_\varepsilon(x)=\varepsilon^{-3}\theta(x/\varepsilon)$.
For any $L^1_{\rm loc}(\R^3)$ function $f$, we denote its mollification by
\[
J_\varepsilon f:=\theta_\varepsilon*f.
\]
Let us consider, for $\varepsilon>0$, the following mollified system
\begin{equation}
\label{MMP}
\left\{
\begin{aligned}
&\partial_t\uep+\P((J_\varepsilon \uep)\cdot \nabla \uep)=(\mu+\chi)\Delta \uep+2\chi\curl  \wep,\\
&\partial_t\wep+(J_\varepsilon \uep)\cdot \nabla \wep=\gamma\Delta \wep+\kappa\nabla(\dive  \wep)
	+2\chi\curl  \uep-4\chi \wep\\
&(\uep,\wep)|_{t=0}=(u_0,w_0)\\
&\dive  \uep=0
\end{aligned}
\right.
 \qquad 
x\in\R^3, \;t>0\\
\end{equation}

Since $J_\varepsilon \uep$ is smooth, the proof of the existence of finite energy solutions of the above system is classical. In addition, we can prove that every finite energy solution of this mollified system verifies a strong energy inequality. More precisely, we have the following proposition.
\begin{proposition}\label{existep}
Let $(u_0,w_0) \in L^2_\sigma(\R^3) \times L^2(\R^3) $ and let $\gamma\geq0$. There exists a weak solution $(\uep,\wep)$ to \eqref{MMP} with initial data $(u_0,w_0)$ such that:
\begin{enumerate}[label=\alph*)]
\item $\uep\in L^\infty(\R_+;L^2(\R^3))\cap C^0_w([0,\infty);L^2(\R^3))\cap L^2(\R_+;\dot H^1(\R^3))$;
\item
\begin{itemize*}
\item $\wep\in L^\infty(\R_+;L^2(\R^3))\cap C^0_w([0,\infty);L^2(\R^3))\cap L^2(\R_+;\dot H^1(\R^3))$ if $\gamma>0$\\
\item $\wep\in L^\infty(\R_+;L^2(\R^3))\cap C^0_w([0,\infty);L^2(\R^3))$ if $\gamma=\kappa=0$;\\
\item $\wep\in L^\infty(\R_+;L^2(\R^3))\cap C^0_w([0,\infty);L^2(\R^3))$ and $\dive\wep\in L^2(\R_+\times\R^3)$ if $\gamma=0$ and $\kappa>0$.
\end{itemize*} 
\item the following energy equality holds true:
\begin{multline}\label{enereqep}
\nl2{\uep(t)}^2+\nl2{\wep(t)}^2+2\mu\int_s^t\nl2{\nabla \uep}^2+2\gamma\int_s^t\nl2{\nabla \wep}^2+2\kappa\int_s^t\nl2{\dive \wep}^2\\
+2\chi\int_s^t\nl2{\curl  \uep-2\wep}^2
=\nl2{\uep(s)}^2+\nl2{\wep(s)}^2
\end{multline}
for all $t\geq s\geq0$.
\end{enumerate}
\end{proposition}
\begin{remark}
We stress the fact that the energy equality above covers the cases $\gamma=0$ or $\kappa=0$, with the convention that 
the integrals preceded by a zero coefficient do not appear.
\end{remark}
\begin{proof}
Since $\jep\uep$ is smooth, the existence of a weak solution with the regularity stated in items a) and b) follows in a classical manner. For the convenience of the reader, we sketch the proof following the approach given in \cite[Proposition~2.1]{DanP}.

For any positive integer $n$, let us introduce the frequency truncation operator $\EE_n$, defined by
\[
\widehat{\EE_n f}(\xi):=\widehat f(\xi){\bf 1}_{ |\xi|\le n}.
\]
We consider the filtered mollified system, for  $x\in\R^3$ and $t>0$:
\begin{equation}
\label{MMP-fil}
\left\{
\begin{aligned}
&\partial_t \unep =-\EE_n\P((J_\varepsilon \P \EE_n \unep )\cdot \nabla \EE_n \unep )
+(\mu+\chi)\Delta \EE_n \unep +2\chi\curl  \EE_n \wnep ,\\
&\partial_t \wnep =-\EE_n((J_\varepsilon\P\EE_n\unep) \cdot \nabla \EE_n \wnep )
+\gamma\Delta\EE_n\wnep
	+\kappa\nabla(\dive  \EE_n \wnep )+2\chi\curl  \EE_n\unep-4\chi \EE_n \wnep .\\
&(\unep ,\wnep )|_{t=0}=(\EE_n u_0,\EE_n w_0).
\end{aligned}
\right.
\end{equation}
If $(\unep ,\wnep )\in \EE_nL^2$, then by standard calculations (see, e.g., \cite[Section~6.1]{BahCD11}) one verifies that all the terms appearing in the right-hand side of the first two equations in~\eqref{MMP-fil} belong to~$\EE_nL^2(\R^3)$. These calculations also imply that the Cauchy--Lipschitz theorem applies in $\EE_nL^2$, endowed with the usual $L^2$-norm.
Hence, the filtered mollified system~\eqref{MMP-fil} admits a maximal time $0<T_n\le+\infty$ and unique maximal solution 
$(\unep ,\wnep )\in C^1([0,T_n),\EE_nL^2(\R^3))$. Applying the projection $\EE_n$ to \eqref{MMP-fil} and using that   $\EE^2_n=\EE_n$ we observe that $(\EE_n \unep,\EE_n \wnep)$ also solves   \eqref{MMP-fil}. By uniqueness we must have that $\EE_n \unep=\unep$ and $\EE_n \wnep=\wnep$ so \eqref{MMP-fil} becomes
\begin{equation}
\label{MMP-fil-bis}
\left\{
\begin{aligned}
&\partial_t \unep =-\EE_n\P((J_\varepsilon \P  \unep )\cdot \nabla  \unep )
+(\mu+\chi)\Delta \unep +2\chi\curl   \wnep ,\\
&\partial_t \wnep =-\EE_n((J_\varepsilon\P\unep) \cdot \nabla  \wnep )
+\gamma\Delta\wnep
	+\kappa\nabla(\dive   \wnep )+2\chi\curl  \unep-4\chi  \wnep .\\
&(\unep ,\wnep )|_{t=0}=(\EE_n u_0,\EE_n w_0).
\end{aligned}
\right.
\end{equation}

Clearly $(\unep ,\wnep )=(\EE_n \unep ,\EE_n \wnep )\in H^\infty(\R^3)$, so scalar multiplication of the first equation by $u_n$ and of the second equation by $w_n$ is allowed. Using again that $(\unep ,\wnep )=(\EE_n \unep ,\EE_n \wnep )$ together with the self-adjointness of $\EE_n$ we get that
$(\unep ,\wnep )$ satisfies the energy equality for $0\le t<T_n$:
\begin{multline}\label{Difi}
\nl2{\unep(t)}^2+\nl2{\wnep(t)}^2+2\mu\int_0^t \nl2{\nabla \unep}^2+2\gamma\int_0^t\nl2{\nabla \wnep}^2+2\kappa\int_0^t\nl2{\div \wnep}^2\\
+2\chi\int_0^t\nl2{\curl  \unep-2\wnep}^2
=\nl2{\EE_n u_0}^2+\nl2{\EE_n w_0}^2
\leq \nl2{u_0}^2+\nl2{w_0}^2.
\end{multline}
In particular, $(\unep ,\wnep )\in L^\infty([0,T_n),\EE_nL^2(\R^3))$, so by the Cauchy-Lipschitz theory  $T_n=+\infty$.

Relation \eqref{Difi} also implies that $\unep$ and $\wnep$ are bounded independently of $n$ (and of $\ep$ too but this is irrelevant at this stage) in the spaces mentioned in the items a) and b) of the statement. We can therefore extract a subsequence, also denoted by 
$(\unep ,\wnep )$ which converges weakly to a couple $(\uep,\wep)$ belonging to the spaces announced in the items a) and b) of the statement. One can pass to the limit $n\to\infty$ in \eqref{MMP-fil-bis} with classical compactness arguments and show that $(\uep,\wep)$ verifies the weak formulation of \eqref{MMP}, as in \cite[Proposition~2.1]{DanP}, where the existence of weak solution for a similar system of equations --- the Boussinesq system without dissipation in the temperature equation, was established. Note that the weak continuity in time of $\uep$ and $\wep$ with values in $L^2$ follows from the boundedness in $L^2$, the continuity in time with values in $\mathscr{D}'$ (which can be easily deduced from the PDE verified by  $\uep$ and $\wep$) and the density of $\mathscr{D}$ in $L^2$.  

It remains to show item c), the energy equality \eqref{enereqep}.
This energy equality will follow from two facts. One fact is that the equation of $\uep$ can be multiplied by $\uep$ yielding
\begin{equation}\label{enerequ}
\nl2{\uep(t)}^2+2(\mu+\chi)\int_s^t\nl2{\nabla \uep}^2-4\chi \int_s^t\int_{\R^3}\wep\cdot(\curl  \uep)= \nl2{\uep(s)}^2
\end{equation}
 for all $t\geq s\geq0$. The second fact is that $\wep$ verifies the energy equality
\begin{equation}\label{enereqw}
\nl2{\wep(t)}^2+2\gamma\int_s^t\nl2{\nabla \wep}^2
+2\kappa\int_s^t\nl2{\div \wep}^2
+8\chi\int_s^t\nl2{\wep}^2-4\chi \int_s^t\int_{\R^3}\wep\cdot(\curl  \uep)=\nl2{\wep(s)}^2
\end{equation}
for all $t\geq s\geq0$.

Adding \eqref{enerequ} and \eqref{enereqw} clearly implies the  energy equality \eqref{enereqep}. It remains to show \eqref{enerequ} and \eqref{enereqw}.

One can easily verify that the equation of $\uep$ can be multiplied by $\uep$ yielding  \eqref{enerequ}. Indeed, we know that $\uep\in  L^2([0,\infty);\dot H^1(\R^3))$ so $\Delta\uep\in  L^2([0,\infty);\dot H^{-1}(\R^3))$. Also $\wep\in L^\infty(\R_+;L^2(\R^3))$ so $\curl\wep\in  L^\infty([0,\infty);\dot H^{-1}(\R^3))$. Since $\jep\uep$ is smooth, we clearly have that $\P((\jep\uep)\cdot\nabla\uep)=\P\dive((\jep\uep)\otimes\uep)\in L^\infty([0,\infty);\dot H^{-1}(\R^3))$. We infer that $\partial_t\uep=-\P((J_\varepsilon \uep)\cdot \nabla \uep)+(\mu+\chi)\Delta \uep+2\chi\curl  \wep\in L^2_{loc}([0,\infty);\dot H^{-1}(\R^3))$. Since all the terms in the equality
\begin{equation*}
\partial_t\uep+\P((J_\varepsilon \uep)\cdot \nabla \uep)=(\mu+\chi)\Delta \uep+2\chi\curl  \wep
\end{equation*}
belong to $L^2_{loc}([0,\infty);\dot H^{-1}(\R^3))$, the relation above can be multiplied by $\uep$ which belongs to the dual space $L^2([0,\infty);\dot H^{-1}(\R^3))$ (see for instance \cite{lions_problemes_1968}). This proves \eqref{enerequ}.

\medskip

To prove \eqref{enereqw} we distinguish three cases.

{\it Case 1}: $\gamma>0$. Since there is a diffusion term in the equation of $\wep$ the same argument as above applies.

{\it Case 2}: $\gamma=\kappa=0$. In this case the proof of \eqref{enereqw} follows the argument given in \cite[Proposition 2.5]{brandolese_2d_2024} where a similar energy equality was proved (see relation (2.5) from that reference) for a similar equation in a more complicated setting (there $\jep \wep$ is replaced by $w$). The equation considered in that reference is scalar but we can work on the equation of $\wep$ from \eqref{MMP} component by component. The proof of \cite[Proposition 2.5]{brandolese_2d_2024} does not depend on the space dimension. Relation  \eqref{enereqw} is therefore a direct consequence of the argument given in the proof of \cite[Proposition 2.5]{brandolese_2d_2024}.

{\it Case 3}: $\gamma=0$ and $\kappa>0$.

We apply $\jal$ to the equation of $\wep$ given in \eqref{MMP}, we multiply it by $\jal\wep$ and integrate on $[s,t]\times\R^3$. We get after some integrations by parts
\begin{multline}\label{eqhep}
\nl2{\jal\wep(t)}^2
+2\kappa \int_s^t\nl2{\jal\dive\wep}^2
+8\chi \int_s^t\nl2{\jal\wep}^2
=\nl2{\jal\wep(s)}^2\\
+4\chi\int_s^t\int_{\R^3}\jal \curl\uep\cdot\jal\wep
-2\int_s^t\int_{\R^3}\jal((\jep\uep)\cdot\nabla \wep)\cdot\jal\wep. 
\end{multline}

We let $\al\to0$. Since $\jal$ is a mollifier and $\wep(t)\in L^2$ we know that $\jal\wep(t)\to\wep(t)$ in $L^2$ as $\al\to0$, so
\begin{equation*}
\lim_{\al\to0}\nl2{\jal\wep(t)}^2=\nl2{\wep(t)}^2.
\end{equation*}
Similarly
\begin{align*}
\lim_{\al\to0}\int_s^t\nl2{\jal\dive\wep}^2&=\int_s^t\nl2{\dive\wep}^2\\
\lim_{\al\to0}\int_s^t\nl2{\jal\wep}^2&=\int_s^t\nl2{\wep}^2\\
\lim_{\al\to0}\nl2{\jal\wep(s)}^2&=\nl2{\wep(s)}^2.
\end{align*}

Next, we have that $\curl\uep$ and $\wep$ belong to the space $L^2\big((s,t)\times\R^3\big)$ so $\jal\curl\uep\to\curl\uep$ and  $\jal\wep\to\wep$ in $L^2\big((s,t)\times\R^3\big)$ as $\al\to0$. Therefore
\begin{equation*}
\lim_{\al\to0}\int_s^t\int_{\R^3}\jal \curl\uep\cdot\jal\wep
=\int_s^t\int_{\R^3} \curl\uep\cdot\wep.
\end{equation*}

We bound the last term of \eqref{eqhep} with a commutator estimate:
\begin{align}
\int_s^t\int_{\R^3}\jal((\jep\uep)\cdot\nabla \wep)\cdot\jal\wep
&=\int_s^t\int_{\R^3}\big[\jal, (\jep\uep)\cdot\nabla\big]\wep\cdot \jal\wep
+\int_s^t\int_{\R^3}(\jep\uep)\cdot\nabla \jal\wep\cdot \jal\wep  \notag\\
&=\int_s^t\int_{\R^3}\big[\jal, (\jep\uep)\cdot\nabla\big]\wep\cdot \jal\wep\label{commutator}\\
&\leq \|\big[\jal, (\jep\uep)\cdot\nabla\big]\wep\|_{L^2((s,t)\times\R^3)}\|\jal\wep\|_{L^2((s,t)\times\R^3)}\notag\\
&\stackrel{\al\to0}\longrightarrow 0\notag
\end{align}
where we used the classical properties of commutators see for example \cite[Lemma 2.4]{brandolese_2d_2024} and the fact that $\jep\uep$ is Lipschitz in the space variable.

We conclude that letting $\al\to0$ in \eqref{eqhep} implies \eqref{enereqw} and completes the proof.
\end{proof}

Following (with slight modification) Lemari\'e-Rieusset's terminology~\cite{Lem02}, we introduce the following notion.
\begin{definition}
Let $(u_0,w_0)\in L^2_\sigma(\R^3) \times L^2(\R^3)$.
\label{def:restricted}
\begin{itemize}
\item A solution $(\uep,\wep)$ of \eqref{MMP} as obtained in Proposition \ref{existep} is called a \textit{mollified solution}.
\item A \emph{restricted Leray solution} of the micropolar system~\eqref{MP} with initial data $(u_0,w_0)$ is a couple $(u,w)$ obtained as a limit in the sense of distributions of a sequence of mollified solutions: there exists $\ep_n\to0$ and a solution  $(\uepn,\wepn)$ of \eqref{MMP} with $\ep=\ep_n$ such that $(\uepn,\wepn)\to (u,w)$ in $\mathscr{D}'\big((0,\infty)\times\R^3\big)$.
\end{itemize}
\end{definition}

It is not hard to show that a restricted Leray solution of \eqref{MP} is a solution in the sense of the distributions, that it verifies an energy estimate and that every bound in $L^2$ for $\uep$ and $\wep$ independent of $\ep$ is also true for the restricted Leray solution. More precisely, we have the following result.
\begin{proposition}\label{prop:exist}
Let $(u_0,w_0)\in L^2_\sigma(\R^3) \times L^2(\R^3)$ and let $\gamma\geq0$. Then there exists a restricted Leray solution of \eqref{MP}. In addition, any restricted Leray solution $(u,w)$ of \eqref{MP} verifies the following properties:
\begin{enumerate}[label=\alph*)]
\item The couple $(u,w)$ solves \eqref{MP} in the sense of the distributions and $u\big|_{t=0}=u_0$, $w\big|_{t=0}=w_0$.
\item $u\in L^\infty(\R_+;L^2(\R^3))\cap C^0_w([0,\infty);L^2(\R^3))\cap L^2(\R_+;\dot H^1(\R^3))$;
\item
\begin{itemize*}
\item $w\in L^\infty(\R_+;L^2(\R^3))\cap C^0_w([0,\infty);L^2(\R^3))\cap L^2(\R_+;\dot H^1(\R^3))$ if $\gamma>0$;\\
\item $w\in L^\infty(\R_+;L^2(\R^3))\cap C^0_w([0,\infty);L^2(\R^3))$ if $\gamma=\kappa=0$;\\
\item $w\in L^\infty(\R_+;L^2(\R^3))\cap C^0_w([0,\infty);L^2(\R^3))$ and $\dive w\in L^2(\R_+\times\R^3)$ if $\gamma=0$ and $\kappa>0$.
\end{itemize*} 
\item the following energy inequality holds true:
\begin{multline}\label{enereq}
\nl2{u(t)}^2+\nl2{w(t)}^2+2\mu\int_0^t\nl2{\nabla u}^2+2\gamma\int_0^t\nl2{\nabla w}^2+2\kappa\int_0^t\nl2{\dive w}^2\\
+2\chi\int_0^t\nl2{\curl  u-2w}^2
\leq\nl2{u_0}^2+\nl2{w_0}^2
\end{multline}
for all $t\geq0$.
\item If $f,g:\R_+\to\R_+$ do not depend on $\ep_n$ and if we have the bounds $\nl2{\uepn(t)}\leq f(t)$ and $\nl2{\wepn(t)}\leq g(t)$ for all $t\geq0$, then the same bounds hold true for the restricted Leray solution $(u,w)$ which is the limit of $(\uepn,\wepn)$: $\nl2{u(t)}\leq f(t)$ and $\nl2{w(t)}\leq g(t)$ for all $t\geq0$.
\item 
Let $\zl=(\ul,\wl)$ be the solution of the linear micropolar system, $\zepn=(\uepn,\wepn)$ and $z=(u,w)$.
If $h:\R_+\to\R_+$ does not depend on $\ep_n$ and if we have the bounds $\nl2{\zepn(t)-\zl(t)}\leq h(t)$ for all $t\geq0$, then  $\nl2{z(t)-\zl(t)}\leq h(t)$ for all $t\geq0$.
\end{enumerate}
\end{proposition}
\begin{proof}
The existence of a restricted Leray solution is trivial. Indeed, thanks to \eqref{enereqep}, any sequence of mollified solutions is bounded in $L^\infty(\R_+;L^2(\R^3))$ so it has a subsequence converging in the sense of the distributions. The limit of such a subsequence is a restricted Leray solution.

Let now $(u,w)$ be a restricted Leray solution and let $\zepn=(\uepn,\wepn)$ be a sequence of mollified solutions  \eqref{MMP} converging in the sense of the distributions to $z=(u,w)$. We write the energy equality \eqref{enereqep} for $s=0$:
\begin{multline}\label{Difi1}
\nl2{\uepn(t)}^2+\nl2{\wepn(t)}^2+2\mu\int_0^t \nl2{\nabla \uepn}^2+2\gamma\int_0^t\nl2{\nabla \wepn}^2+2\kappa\int_0^t\nl2{\div \wepn}^2\\
+2\chi\int_0^t\nl2{\curl  \uepn-2\wepn}^2
= \nl2{u_0}^2+\nl2{w_0}^2.
\end{multline}

This relation implies that  $\uepn$ and $\wepn$ are bounded independently of $\ep_n$ in the spaces mentioned in the items b) and c) of the statement. So necessarily $u$ and $w$ must belong to these spaces (we will justify later the weak continuity in time with values in $L^2$). Moreover we can subtract a subsequence, also denoted by 
$(\uepn ,\wepn )$, which converges weakly to $(u,w)$ in the spaces announced in the items b) and c) of the statement.

From \eqref{MMP} we see that $\partial_t \zepn$ is bounded in $L^\infty(\R_+;H^{-3})$. The Ascoli theorem together with the compact embedding of $H^{-3}$ into $H^{-4}_{loc}$ allow then to extract a sub-sequence, also denoted by $\zepn$, such that $\zepn\to z$ in $C^0\big([0,\infty);H^{-4}_{loc}\big)$. In particular $\zepn(0)\to z(0)$ in $H^{-4}_{loc}$, so $z(0)=z_0$.

The weak continuity in time of $u$ and $w$ with values in $L^2$ follows from the boundedness in $L^2$, the continuity in time with values in $\mathscr{D}'$ (which follows from $u,w\in C^0\big([0,\infty);H^{-4}_{loc}\big)$ as seen above) and the density of $\mathscr{D}$ in $L^2$. This proves items b) and c) of the statement.

We observed that $\uepn$ converges to $u$ weakly in $L^2_{loc}(\R_+;H^1)$ and strongly in  $C^0\big([0,\infty);H^{-4}_{loc}\big)$. The quantity $\jepn\uepn$ is bounded in the same spaces as $\uepn$ so, after extracting another subsequence, enjoys the same convergence properties. Clearly $\jepn\uepn-\uepn$ goes to 0 in the sense of the distributions, so $\jepn\uepn$ and $\uepn$ have the same limit. We infer that $\jepn\uepn$ converges to $u$ weakly in $L^2_{loc}(\R_+;H^1)$ and strongly in  $C^0\big([0,\infty);H^{-4}_{loc}\big)$, so $\jepn\uepn\to u$ strongly in $L^2_{loc}(\R_+\times\R^3)$. Recalling that $\wepn$ converges to $w$ weak$*$ in $L^\infty(\R_+;L^2)$, one can pass to the limit $\varepsilon_n\to0$ in \eqref{MMP} (where $\ep$ is replaced by $\varepsilon_n$) and obtain that the couple $(u,w)$ verifies the micropolar system \eqref{MP} in the sense of the distributions. We already observed that $z(0)=(u_0,w_0)$, so item a) is completely proved.

We prove now the energy inequality \eqref{enereq}. Observe first that the strong convergence $\zepn\to z$ in $C^0\big([0,\infty);H^{-4}_{loc}\big)$ implies the time-pointwise convergence $\zepn(t)\to z(t)$ in $H^{-4}_{loc}$, so in $\mathscr{D}'$ too, for all $t\geq0$. Recalling that $\zepn(t)$ is also bounded in $L^2$, we infer that $\zepn(t)\to z(t)$ weakly in $L^2$ for  all $t\geq0$. The weak lower semi-continuity of norms implies that
\begin{equation*}
\nl2{z(t)}\leq \liminf_{n\to\infty}\nl2{\zepn(t)}
\end{equation*}
for all $t\geq0$. Similarly
\begin{equation*}
\int_0^t \nl2{\nabla u}^2
\leq \liminf_{n\to\infty}\int_0^t \nl2{\nabla \uepn}^2
\end{equation*}
and
\begin{equation*}
\int_0^t\nl2{\curl  u-2w}^2
\leq \liminf_{n\to\infty}\int_0^t\nl2{\curl  \uepn-2\wepn}^2
\end{equation*}
for all $t\geq0$. In addition, if $\gamma>0$ we also have that
\begin{equation*}
\int_0^t \nl2{\nabla w}^2
\leq \liminf_{n\to\infty}\int_0^t \nl2{\nabla \wepn}^2
\end{equation*}
and  if $\kappa>0$ we  have that
\begin{equation*}
\int_0^t \nl2{\dive w}^2
\leq \liminf_{n\to\infty}\int_0^t \nl2{\dive \wepn}^2.
\end{equation*}

We conclude that taking the $\liminf_{n\to\infty}$ in \eqref{Difi1} implies the energy inequality \eqref{enereq}.

We prove now item e). Assume that $\nl2{\uepn(t)}\leq f(t)$. We take  the $\liminf\limits_{n\to\infty}$ and use the weak lower semi-continuity of norms together with the fact that $\uepn(t)$ converges weakly in $L^2$ to $u(t)$ for all $t\geq0$:
\begin{equation*}
\nl2{u(t)}\leq \liminf_{n\to\infty} \nl2{\uepn(t)}\leq f(t).
\end{equation*}
Similarly, the bound  $\nl2{\wepn(t)}\leq g(t)$ implies the bound  $\nl2{w(t)}\leq g(t)$.

Item f) follows in the same manner since we of course have that $\zepn(t)-\zl(t)\to z(t)-\zl(t)$ weakly in $L^2$ for all $t\geq0$. This completes the proof.
\end{proof}

\section{Explicit formulas and precise asymptotic behavior for the linear part}
\label{linearsect}

The linear part of the micropolar system reads:
\begin{equation}
\label{LMP}
\left\{
\begin{aligned}
&\dert \ul =(\mu+\chi)\Delta\ul +2\chi\curl  \wl, \qquad 
x\in\R^3, \;t>0\\
&\dert \wl =\gamma\Delta \wl+\kappa\nabla(\dive  \wl)+2\chi\curl  \ul-4\chi \wl.\\
&\dive  \ul=0
\end{aligned}
\right.
\end{equation}
We define $\zl:=(\ul,\wl)$.

At the linear level, the two equations are not too strongly coupled. To take advantage of this fact we consider the Helmholtz decomposition
\[
\wl =\P \wl +(\wl -\P \wl )=\hl -\nabla q, \qquad\text{with }\hl :=\P \wl .
\]
Taking the divergence in the first equation above yields
\begin{equation}\label{eqq}
\Delta q=-\div \wl . 
\end{equation}
Taking the divergence in the second equation of \eqref{LMP} yields
\begin{equation*}
\partial_t(\div \wl )=(\gamma+\kappa)\Delta\div \wl -4\chi\div \wl 
\end{equation*}
which implies exponential decay for $\div \wl $. Indeed, taking the Fourier transform above results in 
\begin{equation*}
\partial_t(\widehat{\div \wl })=\bigl[-4\chi-(\gamma+\kappa)|\xi|^2\big]\widehat{\div \wl }
\end{equation*}
which can be solved explicitly:
\begin{equation*}
\widehat{\div \wl} =e^{-t(4\chi+(\gamma+\kappa)|\xi|^2)}\widehat{\div w_0}
=e^{-t(4\chi+(\gamma+\kappa)|\xi|^2)}i\xi\cdot\widehat{w_0}.
\end{equation*}
From \eqref{eqq} we get that
\begin{equation*}
\nabla q=-\nabla\Delta^{-1}\div \wl 
\end{equation*}
so
\begin{equation*}
\widehat{\nabla q}=i\frac{\xi}{|\xi|^2}\widehat{\div \wl }
=-e^{-4\chi t}e^{-(\gamma+\kappa)t|\xi|^2}\frac{\xi(\xi\cdot\widehat{w_0})}{|\xi|^2}. 
\end{equation*}
Then
\begin{equation}\label{soldivw}
\widehat \wl =\widehat \hl -\widehat{\nabla q}=\widehat \hl +e^{-4\chi t}e^{-(\gamma+\kappa)t|\xi|^2}\frac{\xi(\xi\cdot\widehat{w_0})}{|\xi|^2}.
\end{equation}

It suffices now to compute $\ul $ and $\hl $. Since $\curl \nabla q=0$ we have that $\curl  \wl =\curl  \hl $. Applying the Leray projector $\P$ to the equation of $\wl $ in \eqref{LMP} implies the following system for $(\ul ,\hl )$:
\begin{equation}
\label{sysL}
\begin{cases}
\dert \ul =(\mu+\chi)\Delta \ul +2\chi\curl  \hl \\
\dert \hl =\gamma\Delta \hl +2\chi\curl  \ul -4\chi \hl \\
\dive  \ul =0\\
\dive  \hl =0.
\end{cases}
\end{equation}
We can further simplify this system by setting $\Omega_L:=\curl  \ul $. Taking the curl of the first PDE above yields
\begin{equation}
\label{LVMP}
\begin{cases}
\dert \Omega_L=(\mu+\chi)\Delta \Omega_L-2\chi\Delta \hl \\
\dert \hl =\gamma\Delta \hl +2\chi\Omega_L-4\chi \hl \\
\dive \Omega_L=0\\
\dive  \hl =0.
\end{cases}
\end{equation}
To solve~\eqref{LVMP}, with initial data $\Omega_L^0$ and $h_0$, we take the Fourier transform and we get
\begin{equation}
\label{FVL}
\begin{pmatrix}
\dert \widehat \Omega_L(\xi,t)\\
\dert \widehat \hl (\xi,t)
\end{pmatrix}
=
\begin{pmatrix}
-(\mu+\chi)|\xi|^2 I_3 & 2\chi|\xi|^2 I_3\\
2\chi I_3 & -(\gamma|\xi|^2 +4\chi)I_3
\end{pmatrix}
\begin{pmatrix}
\widehat \Omega_L(\xi,t)\\
\widehat \hl (\xi,t)
\end{pmatrix}
\end{equation}
Here $I_3$ denotes the $(3\times3)$ identity matrix.

Let $L\colon {\rm Mat}_{2\times 2}(\R)\to{\rm Mat}_{6\times 6}(\R)$ be the linear transformation
\[
A=\begin{pmatrix} a_{11}&a_{12}\\a_{21}&a_{22}\end{pmatrix}
\mapsto L(A)
=
\begin{pmatrix}
a_{11}I_3&a_{12}I_3\\a_{21}I_3&a_{22}I_3
\end{pmatrix}.
\] 
One easily proves by induction that, for any natural integer $k$, $L(A^k)=L(A)^k$,
and as $L$ is a linear continuous mapping, we get
\[
\exp(L(A))=L(\exp(A)).
\]
This means that the solution of the linear system~\eqref{FVL} is
\begin{equation}
\label{SFV}
\begin{pmatrix}
\widehat \Omega_L(\xi,t)\\
\widehat \hl (\xi,t)
\end{pmatrix}
(\xi,t)
=L\left(\exp \left[ t
\begin{pmatrix}
-(\mu+\chi)|\xi|^2  & 2\chi|\xi|^2 \\
2\chi  & -(\gamma|\xi|^2 +4\chi)
\end{pmatrix}\right]
\right)
\begin{pmatrix}
\widehat \Omega_L^0(\xi)\\
\widehat h_0(\xi)
\end{pmatrix}.
\end{equation}
The exponential of the $2\times2$ matrix above was computed in \cite[Lemma 5.1]{brandolese_2d_2024} using the Putzer spectral formula. The result is the following expression:
\begin{multline}\label{exp1}
\exp \left[ t
\begin{pmatrix}
-(\mu+\chi)|\xi|^2  & 2\chi|\xi|^2 \\
2\chi  & -(\gamma|\xi|^2 +4\chi)
\end{pmatrix}\right]=\\
=\frac1{2\sqrt D}
\begin{pmatrix}
 e^{-t(\alpha-\sqrt D)}(\sqrt D-\beta)+e^{-t(\alpha+\sqrt D)}(\sqrt D+\beta)&
2\chi R \big(e^{-t(\alpha-\sqrt D)}-e^{-t(\alpha+\sqrt D)}\big)\\
2\chi \big(e^{-t(\alpha-\sqrt D)}-e^{-t(\alpha+\sqrt D)}\big)&
e^{-t(\alpha-\sqrt D)}(\sqrt D+\beta)+e^{-t(\alpha+\sqrt D)}(\sqrt D-\beta)
\end{pmatrix},
\end{multline}
where we used the notation
\begin{align*}
\alpha&=\frac12(\mu+\chi+\gamma)R+2\chi\\
\beta&=\frac12(\mu+\chi-\gamma)R-2\chi\\
 D &=\beta^2+4\chi^2R\\
R&=|\xi|^2.
\end{align*}

We know by the Biot-Savart law that
\begin{equation}\label{BS}
\widehat \ul (\xi,t)=\frac{i\xi}{|\xi|^2}\times\widehat\Omega_L(\xi,t).
\end{equation}

Putting together relations \eqref{SFV}, \eqref{exp1} and \eqref{BS} imply the following formulas for $\widehat \ul $ and $\widehat \hl $:
\begin{equation}
\label{solou2}
\left\{\begin{aligned}
\widehat \ul (\xi,t)&= E_{1,1}(\xi,t)\widehat u_0(\xi)+E_{2,1}(\xi,t)i\xi\times\widehat h_0(\xi)\\
\widehat \hl (\xi,t)&= E_{2,1}(\xi,t)i\xi\times\widehat u_0(\xi)+E_{2,2}(\xi,t)\widehat h_0(\xi)\\
\end{aligned}
\right.
\end{equation}
where
\begin{align}
E_{1,1}&=\frac{1}{2\sqrt{ D }}\bigl[e^{-(\alpha-\sqrt{ D })t}(-\beta+\sqrt{ D })+e^{-(\alpha+\sqrt{ D })t}(\beta+\sqrt{ D })\bigr]\notag\\
E_{2,1}&=\frac{\chi}{\sqrt D }e^{-(\alpha-\sqrt{ D })t}(1-e^{-2t\sqrt{ D }})\notag\\
E_{2,2}&=\frac{1}{2\sqrt{ D }}\bigl[e^{-(\alpha-\sqrt{ D })t}(\beta+\sqrt{ D })+e^{-(\alpha+\sqrt{ D })t}(-\beta+\sqrt{ D })\bigr].\label{e22}
\end{align}

Using \eqref{soldivw} and \eqref{solou2} we conclude that the solution of \eqref{LMP} is given by 
\begin{equation}
\label{solou3}
\left\{\begin{aligned}
\widehat \ul (\xi,t)&= E_{1,1}(\xi,t)\widehat u_0(\xi)+E_{2,1}(\xi,t)i\xi\times\widehat h_0(\xi)\\
\widehat \wl (\xi,t)&= E_{2,1}(\xi,t)i\xi\times\widehat u_0(\xi)+E_{2,2}(\xi,t)\widehat h_0(\xi)+e^{-4\chi t}e^{-(\gamma+\kappa)t|\xi|^2}\frac{\xi(\xi\cdot\widehat{w_0})}{|\xi|^2}.
\end{aligned}
\right.
\end{equation}

The quantities $E_{1,1}$, $E_{2,1}$ and $E_{2,2}$ have been introduced and studied in \cite{brandolese_2d_2024}, see relations (5.6)--(5.8) of that paper. We will rely on the estimates from that paper.

Let  $M$ be the linear operator associated with the linear problem~\eqref{LMP}, so that 
$\zl=e^{tM}z_0$.
Let $K(\cdot,t)$ be the symbol of the linear micropolar operator $e^{tM}$ so that 
\begin{equation}\label{notationsgr}
\widehat \zl(\xi,t)=K(\xi,t)\widehat z_0(\xi)
=\mathcal{F}(e^{tM}z_0)(\xi).
\end{equation}
Since $\zl$ has 6 components, $K$ is a $6\times 6$ matrix. Thanks to \eqref{solou3}, \eqref{soldivw} and using that $\xi\times\widehat h_0(\xi)=\xi\times\widehat w_0(\xi)$, one can show the following explicit formula for $K$:
\begin{equation}
\label{formulaK}
K(\xi,t)=
\begin{pmatrix}
E_{11}(\xi,t)I_3&iE_{21}(\xi,t)\ \xi\times\\
iE_{21}(\xi,t)\ \xi\times&E_{22}(\xi,t) I_3
\end{pmatrix}
+
\begin{pmatrix}
O_3&O_3\\
O_3&e^{-4\chi t}e^{-(\gamma+\kappa)t|\xi|^2} (1-E_{22}(\xi,t)) \frac{\xi\otimes\xi}{|\xi|^2}
\end{pmatrix}
\end{equation}
where $I_3$ denotes the $3\times3$ identity, $O_3$ is the $3\times3$ matrix with zero entries, $\xi\otimes\xi$ is the matrix $(\xi_i\xi_j)_{1\leq i,j\leq3}$ and $\xi\times$ is the matrix whose action on a vector is the exterior product with $\xi$:
\begin{equation*}
\xi\times=\begin{pmatrix}
0 & -\xi_3 & \xi_2 \\
\xi_3 & 0 & -\xi_1 \\
-\xi_2 & \xi_1 & 0
\end{pmatrix}
.
\end{equation*}

We observe that the first term on the right-hand side of \eqref{formulaK} is quite similar to the formula for the symbol $K$ from \cite[relation (5.9)]{brandolese_2d_2024}, so we will be able to use the estimates proved in that reference to bound this first term. The second term on the right-hand side of \eqref{formulaK} is exponentially decaying as $t\to\infty$ because one can easily check from \eqref{e22} that $E_{22}$ is uniformly bounded in time and space. Therefore, this second term is negligible. 

The same proof as in  \cite[Proposition 5.3]{brandolese_2d_2024} shows the following bound for $K$:
\begin{lemma}
There exists $c>0$ such that 
\begin{equation*}
|K(\xi,t)| \lesssim
\begin{cases}  
e^{-c|\xi|^2t} &\text{if $\gamma>0$}\\
\max\{e^{-ct},e^{-ct|\xi|^2}\}\ &\text{if $\gamma=0$}.
\end{cases}
\end{equation*}
In particular
\begin{equation}
\label{infoK}
\left\{
\begin{aligned}
& \esssup_{\xi\in\R^3,\,t\ge0}|K(\xi,t)|<\infty\\
& \text{for a.e. $\xi\in\R^3$}, \quad K(\xi,t)\to0\quad\text{as $t\to+\infty$}.
\end{aligned}
\right.
\end{equation}
\end{lemma}
A second estimate that we require is similar to \cite[Lemma 6.4]{brandolese_2d_2024} and the proof is the same with the appropriate modifications due to the change of dimension:
\begin{lemma}

Let $\Gamma\ge0$ and 
$z_0\in L^2(\R^3)^6$, $\dive u_0=0$. Assume also that
$\|\zl(t)\|_{L^2}^2=O(t^{-\Gamma})$ as $t\to+\infty$. 
In the case $\gamma=0$ we additionally assume $\int (1+|\xi|)|\widehat z_0(\xi)|\dd\xi <\infty$.
Then we have
\begin{equation}
\label{AA}
\begin{split}
\|\zl(t)\|_{L^\infty}^2+\|\nabla \zl(t)\|_{L^\infty}^2=O(t^{-3/2-\Gamma}),
\end{split}
\qquad\text{as $t\to+\infty$}.
\end{equation}
\end{lemma}

The precise asymptotic behavior of the quantities $E_{1,1}$, $E_{2,1}$ and $E_{2,2}$ have been studied in \cite{brandolese_2d_2024}. We collect in the following proposition the estimates proved in Propositions 5.5, 5.6 and 5.7 from \cite{brandolese_2d_2024}.
\begin{proposition}\label{prop-e}
There exist two positive constants $c=c(\mu,\chi,\gamma)$ and $C=C(\mu,\chi,\gamma)$ such that:
\begin{enumerate}[label=\alph*)]
\item if $\gamma>0$ then 
\begin{align*}
|E_{1,1}(\xi,t)-e^{-\mu t|\xi|^2}|&\leq \frac{C}{t}e^{-ct|\xi|^2}\\
\big|E_{2,2}(\xi,t)-\frac{|\xi|^2}4e^{-\mu t|\xi|^2}\big|&\leq \frac{C}{t^2}e^{-c t|\xi|^2}\\
\big|E_{2,1}(\xi,t)-\frac12e^{-\mu t|\xi|^2}\big|&\leq \frac{C}{t}e^{-c t|\xi|^2}
\end{align*}
for all $ t\geq 1$ and $\xi\in\R^3$.
\item if $\gamma=0$ then 
\begin{align*}
|E_{1,1}(\xi,t)-e^{-\mu t|\xi|^2}|&\leq \frac{C}{t}e^{-ct|\xi|^2}+\frac{C}{1+|\xi|^2}e^{-ct}\\
\big|E_{2,2}(\xi,t)-\frac{|\xi|^2}4e^{-\mu t|\xi|^2}\big|&\leq \frac{C}{t^2}e^{-c t|\xi|^2}+Ce^{-c t}\\
\big|E_{2,1}(\xi,t)-\frac12e^{-\mu t|\xi|^2}\big|&\leq \frac{C}{t}e^{-c t|\xi|^2}+\frac{C}{1+|\xi|^2}e^{-c t}
\end{align*}
for all $ t\geq 1$ and $\xi\in\R^3$.
\end{enumerate}
\end{proposition}

We can now prove the main result of this section.
\begin{theorem}
Assume that $(u_0,w_0)\in L^2_\sigma(\R^3) \times L^2(\R^3)$. The solution of \eqref{LMP} has the following asymptotic behavior in $L^2$:
\begin{equation*}
\big\|\ul -e^{\mu t\Delta}u_0-\frac12\curl e^{\mu t\Delta}w_0\big\|_{L^2}\leq \frac Ct \|u_0\|_{L^2}+\frac C{t^{\frac32}} \|\P w_0\|_{L^2}
\qquad\forall t\geq1
\end{equation*}
and 
\begin{equation*}
\big\|\wl -\frac12\curl e^{\mu t\Delta}u_0+\frac14\Delta e^{\mu t\Delta}\P w_0\big\|_{L^2}\leq\frac C{t^{\frac32}} \|u_0\|_{L^2}+\frac C{t^2} \|w_0\|_{L^2}
\qquad\forall t\geq1
\end{equation*}
for some constant $C$ depending only on the material coefficients.
\end{theorem}
\begin{proof}
To estimate $\ul $ we use the first line of \eqref{solou3} together with Proposition \ref{prop-e}. Observing also that $\curl h_0=\curl w_0$, we write by the Plancherel theorem, 
\begin{align*}
\big\|\ul -e^{\mu t\Delta}u_0-\frac12\curl e^{\mu t\Delta}w_0\big\|_{L^2} & 
=\big\|\ul -e^{\mu t\Delta}u_0-\frac12\curl e^{\mu t\Delta}h_0\big\|_{L^2}  \\
&=C\big\|\widehat \ul -e^{-\mu t|\xi|^2}\widehat{u_0}-\frac12e^{-\mu t|\xi|^2}i\xi\times \widehat{h_0}\big\|_{L^2}\\
&=C\big\|(E_{1,1}-e^{-\mu t|\xi|^2})\widehat{u_0}+(E_{2,1}-\frac12e^{-\mu t|\xi|^2})i\xi\times \widehat{h_0}\big\|_{L^2}\\
&\leq C\sup_\xi|E_{1,1}-e^{-\mu t|\xi|^2}|\, \|\widehat{u_0}\|_{L^2}+C\sup_\xi(|\xi||E_{2,1}-\frac12e^{-\mu t|\xi|^2}|)\|\widehat{h_0}\|_{L^2}.
\end{align*}
For both cases $\gamma=0$ and $\gamma>0$  we use Proposition \ref{prop-e} to bound
\begin{equation*}
C\sup_\xi|E_{1,1}-e^{-\mu t|\xi|^2}|
\leq \frac Ct \sup_\xi e^{-ct|\xi|^2} +\sup_\xi\frac{C}{1+|\xi|^2}e^{-ct}\leq \frac Ct 
\end{equation*}
and
\begin{equation*}
\sup_\xi(|\xi||E_{2,1}-\frac12e^{-\mu t|\xi|^2}|)
\leq \frac Ct \sup_\xi (|\xi|e^{-ct|\xi|^2})
+\sup_\xi \frac{C|\xi|}{1+|\xi|^2}e^{-c t}
\leq \frac C{t^{\frac32}}
\end{equation*}
so
\begin{equation*}
\big\|\ul -e^{\mu t\Delta}u_0-\frac12\curl e^{\mu t\Delta}w_0\big\|_{L^2} \leq \frac Ct \|u_0\|_{L^2}+\frac C{t^{\frac32}} \|h_0\|_{L^2}.
\end{equation*}

Similarly, we use  the second line of \eqref{solou3} together with Proposition \ref{prop-e} to estimate
\begin{align*}
\big\|\wl -\frac12&\curl e^{\mu t\Delta}u_0+\frac14\Delta e^{\mu t\Delta}\P w_0\big\|_{L^2}  
=C\big\|\widehat \wl -\frac12e^{-\mu t|\xi|^2}i\xi\times \widehat{u_0}-\frac14|\xi|^2 e^{-\mu t|\xi|^2}\widehat{h_0}\big\|_{L^2} \\
&=C\big\|(E_{2,1}-\frac12e^{-\mu t|\xi|^2})i\xi\times \widehat{u_0}+(E_{2,2}-\frac14|\xi|^2 e^{-\mu t|\xi|^2})\widehat{h_0}+e^{-4\chi t}e^{-(\gamma+\kappa)t|\xi|^2}\frac{\xi(\xi\cdot\widehat{w_0})}{|\xi|^2}\big\|_{L^2} \\
&\leq C\sup_\xi(|\xi||E_{2,1}-\frac12e^{-\mu t|\xi|^2}|)\, \|\widehat{u_0}\|_{L^2}+C\sup_\xi|E_{2,2}-\frac14|\xi|^2e^{-\mu t|\xi|^2}|\, \|\widehat{w_0}\|_{L^2}+Ce^{-4\chi t}\|\widehat{w_0}\|_{L^2}\\
&\leq \frac C{t^\frac32}\|\widehat{u_0}\|_{L^2}+\big[\frac C{t^2} \sup_\xi e^{-ct|\xi|^2}+Ce^{-c t}\big]\|\widehat{w_0}\|_{L^2}+Ce^{-4\chi t}\|\widehat{w_0}\|_{L^2}\\
&\leq \frac C{t^{\frac32}} \|u_0\|_{L^2}+\frac C{t^2} \|w_0\|_{L^2}.
\end{align*}
where we also used that $h_0=\P w_0$ and that $\P$ is an $L^2$ orthogonal projection.
\end{proof}

\section{Proof of the main result}
\label{sec:decay}

The aim of this section is to prove Theorem~\ref{th:decay}. We will show that the decay estimates stated in that theorem hold true for mollified solutions with constants independent of $\ep$, then we will pass to the limit $\ep\to0$.

We first collect in the following lemma some energy estimates verified by mollified solutions.
\begin{lemma}
Let $\zep=(\uep,\wep)$ 
be a mollified solution to the mollified problem~\eqref{MMP}, as obtained in Proposition~\ref{existep}.
There exist some constants $c_1,\dots,c_4>0$, depending only on the material coefficients, such that
\begin{equation}\label{EI}
\|\uep(t)\|_{L^2}^2+\|\wep(t)\|_{L^2}^2
+c_1\int_s^t \|\nabla \uep\|_{L^2}^2
+c_2\int_s^t\|\wep\|_{L^2}^2
\le \|\uep(s)\|_{L^2}^2+\|\wep(s)\|_{L^2}^2
\end{equation}
and 
\begin{multline}
\label{EID2}
 \|(\zep-\zl)(t)\|_{L^2}^2- \|(\zep-\zl)(s)\|_{L^2}^2
 + c_3\int_s^t\Bigl( 
 \|\nabla (\uep-\ul)\|_{L^2}^2+\|(\wep-\wl)\|_{L^2}^2\Bigr)\\
 \le c_4\int_s^t \|\uep\|_{L^2}^2\Bigl(\|\ul\|_{L^\infty}^2 +\|\nabla \wl\|_{L^\infty}^2\Bigr)  
\end{multline}
for all $t\geq s\geq0$.
\end{lemma}
\begin{proof}
From the energy equality~\eqref{enereqep} we deduce, for all $0\le s\le t$,
\begin{multline}\label{E1}
\|\uep(t)\|_{L^2}^2+\|\wep(t)\|_{L^2}^2  +2(\mu+\chi)\int_s^t\|\nabla \uep\|_{L^2}^2+2\gamma\int_s^t\|\nabla \wep\|_{L^2}^2
+8\chi\int_s^t\|\wep\|_{L^2}^2\\
+2\kappa\int_s^t\|\div \wep\|_{L^2}^2 
\le   \|\uep(s)\|_{L^2}^2+\|\wep(s)\|_{L^2}^2 +8\chi\int_s^t\|\wep\|_{L^2}\,\|\nabla \uep\|_{L^2}.
\end{multline}
But, for any $\eta>0$, we have $2\|\wep\|_{L^2}\,\|\nabla \uep\|_{L^2}\le \frac{1}{\eta}\|\wep\|_{L^2}^2+\eta\|\nabla \uep\|_{L^2}^2$.
We choose  $\frac12< \eta<\frac{\mu+\chi}{2\chi}$, for example,
\[
\eta=\textstyle\frac14(1+\frac{\mu+\chi}{\chi}).
\]
Dropping in \eqref{E1} the positive terms with coefficients $\gamma$ and $\kappa$ implies \eqref{EI} with constants $c_1,c_2>0$ given by
\begin{equation*}
c_1=2(\mu+\chi-2\chi\eta)\quad\text{and}\quad
c_2=2\chi(4-2/\eta).
\end{equation*}

We prove next \eqref{EID2}. Taking the difference between the mollified system \eqref{MMP} and \eqref{LMP}, we obtain
\begin{equation*}
\left\{
\begin{aligned}
&\dert (\uep-\ul)+\P((\jep\uep)\cdot \nabla \uep)=(\mu+\chi)\Delta (\uep-\ul)
	+2\chi\curl  (\wep-\wl),\\
&\dert (\wep-\wl)+(\jep\uep)\cdot\nabla\wep=\gamma\Delta (\wep-\wl)+\kappa\nabla(\dive  (\wep-\wl))
+2\chi\curl  (\uep-\ul)-4\chi (\wep-\wl),\\
&\dive  (\uep-\ul)=0.
\end{aligned}
\right.
\end{equation*}
We will multiply the first equation by $\uep-\ul$ and the second equation by $\wep-\wl$, add the result and do similar estimates as in the first part of the proof. There are three different things to deal with: to justify that the equation of $\uep-\ul$ can be multiplied by $\uep-\ul$, to justify that the equation of $\wep-\wl$ can be multiplied by $\wep-\wl$ and to estimate the nonlinear terms.

The fact that the equation of $\uep-\ul$ can be multiplied by $\uep-\ul$ follows as in the proof of \eqref{enerequ}. Similarly, the fact that the equation of $\wep-\wl$ can be multiplied by $\wep-\wl$ follows as in the proof of \eqref{enereqw}.

Now we deal with the nonlinear terms. There are two of them. The first one is the following:
\[
\begin{split}
\Bigl|\int ((\jep\uep)\cdot\nabla \uep)\cdot(\uep-\ul)\dd x\Bigr|
&=\Bigl|\int ((\jep\uep)\cdot\nabla \uep)\cdot\ul\Bigr|\\
&=\Bigl|\int ((\jep\uep)\cdot\nabla (\uep-\ul))\cdot\ul\Bigr|\\
&\le \|\jep\uep\|_{L^2}\,\|\nabla(\uep-\ul)\|_{L^2}\,\|\ul\|_{L^\infty}\\
&\le \frac{c_1}{4}\,\|\nabla (\uep-\ul)\|_{L^2}^2+\frac{1}{c_1}\|\uep\|_{L^2}^2\|\ul\|_{L^\infty}^2.
\end{split}
\] 
We estimate the second nonlinear term in a different way, relying in a deeper way on the damping in the equation of~$\wep$: 
\begin{multline}\label{weplmult}
\Big|\int ((\jep\uep)\cdot\nabla \wep)\cdot (\wep-\wl)\Big| =\Big|\int((\jep\uep)\cdot\nabla\wl)\cdot(\wep-\wl)\Big|\\
\le \|\jep\uep\|_{L^2} \|\nabla \wl\|_{L^\infty}\,\|\wep-\wl\|_{L^2}
\le \frac{c_2}{4}\|\wep-\wl\|_{L^2}^2+\frac{1}{c_2}\|\uep\|_{L^2}^2\|\nabla\wl\|_{L^\infty}^2.
\end{multline}

In fact, the first term in the relation above is not a convergent integral. But if we recall the proof of \eqref{enereqw}, this term is actually equal to
\begin{multline*}
\lim_{\al\to0}\int \jal((\jep\uep)\cdot\nabla \wep)\cdot \jal(\wep-\wl)
=\lim_{\al\to0}\Big(\int \jal((\jep\uep)\cdot\nabla (\wep-\wl))\cdot \jal(\wep-\wl)\\
+\int \jal((\jep\uep)\cdot\nabla \wl)\cdot \jal(\wep-\wl)\Big).
\end{multline*}
The same commutator estimate as in \eqref{commutator} shows that the first term on the right-hand side above goes to 0 as $\al\to0$, while the second term passes to the limit easily. Then
\begin{equation*}
\lim_{\al\to0}\int \jal((\jep\uep)\cdot\nabla \wep)\cdot \jal(\wep-\wl)
=\int((\jep\uep)\cdot\nabla\wl)\cdot(\wep-\wl)
\end{equation*}
which justifies the first equality in \eqref{weplmult}.

The same calculations as in the proof of \eqref{EI} together with the two estimates of the nonlinear terms imply the following inequality:
\begin{multline*}
\|(\zep-\zl)(t)\|_{L^2}^2
+c_1\int_s^t \|\nabla (\uep-\ul)\|_{L^2}^2
+c_2\int_s^t\|(\wep-\wl)\|_{L^2}^2
\le \|(\zep-\zl)(s)\|_{L^2}^2\\
+\frac{c_1}{2}\int_s^t\|\nabla (\uep-\ul)\|_{L^2}^2+\frac{2}{c_1}\int_s^t\|\uep\|_{L^2}^2\|\ul\|_{L^\infty}^2
+\frac{c_2}{2}\int_s^t\|\wep-\wl\|_{L^2}^2+\frac{2}{c_2}\int_s^t\|\uep\|_{L^2}^2\|\nabla\wl\|_{L^\infty}^2.
\end{multline*}
Relation \eqref{EID2} now follows with constants $c_3,c_4>0$  given by
\[
 c_3=\min\{\textstyle\frac{c_1}{2},\frac{c_2}{2}\}
 \quad\text{and}\quad
 c_4=\max\{\frac{2}{c_1},\frac{2}{c_2}\}.
\]
\end{proof}

The main part of the proof of Theorem~\ref{th:decay} consists in showing the following $\ep$-independent decay estimates for mollified solutions. 
\begin{proposition}
\label{prop:decay}
Let $\mu,\chi>0$ and $\gamma,\kappa\ge0$.
Let $(u_0,w_0)\in L^2_\sigma(\R^3)\times L^2(\R^3)$, $\varepsilon>0$
and $\zep=(\uep,\wep)$ be a mollified solution as constructed in~Proposition~\ref{existep}.
Then we have the following
\begin{itemize}
\item[i)] $\|\zep(t)\|_{L^2}\to0$ as $t\to+\infty$ uniformly with respect to $\varepsilon$.
\item[ii)] If $0\le \Gamma\le 5/2$ and if there exists $C_\Gamma>0$ such that the solution $\zl$ of the linear problem~\eqref{LMP}
starting from $z_0$ satisfies $\|\zl(t)\|_{L^2}^2\le C_\Gamma(1+t)^{-\Gamma}$,
 then $\|\zep(t)\|_{L^2}^2\le C(1+t)^{-\Gamma}$,
 for some constant $C=C(\mu,\chi,C_\Gamma,\|z_0\|_{L^2})>0$ independent of $t$ and of $\varepsilon$.
\item[iii)] Under the condition of the previous item (when $\gamma=0$ we also require the additional
condition $\int (1+|\xi|)|\widehat z_0(\xi)|\dd\xi <\infty $), we have, for some constant $C=C(\mu,\gamma,\chi,C_\Gamma,\|z_0\|_{L^2})>0$
(depending also on $\int (1+|\xi|)|\widehat z_0(\xi)|\dd\xi$
in the case $\gamma=0$), independent of $t$ and of $\varepsilon$,
	\[
	\|(\zep-\zl)(t)\|_{L^2}^2\le 
	\begin{cases}
	C(1+t)^{-1/2}\zeta(t) &\text{if $\Gamma=0$}\\
	C(1+t)^{-1/2-2\Gamma}& \text{if $0<\Gamma<1$}\\
	C(1+t)^{-5/2}(\log(e+t))^2 &\text{if $\Gamma=1$}\\
	C(1+t)^{-5/2}&\text{if $1<\Gamma\le 5/2$}
	\end{cases}
	\]
	with $\lim_{t\to+\infty}\zeta(t)=0$ uniformly with respect to $\varepsilon$.
\end{itemize} 
\end{proposition}

\begin{proof}[Proof of Proposition~\ref{prop:decay}]
In this proof,  we denote the constants appearing in our estimates that are independent of~$t$ and $\varepsilon$ with the same letter $C$, even though they can change from line to line.
All the constants in this proof will be independent of~$\kappa\ge0$, but they may depend on the other parameters $\mu,\gamma,\chi$ and on the initial datum, as specified in the statement of the proposition. The proof proceeds in several steps.

\paragraph{Step 1. Fourier splitting.}

Let $0<\delta\le \min(c_1,c_3)$.
Let us introduce the function
\begin{equation}
\label{choi-g}
g(t):=\sqrt{\textstyle\frac{10}{\delta}}(1+t)^{-1/2}.
\end{equation}
Let $t_0$ be large enough, so that $t_0 \ge 10/\delta$ and $c_2(1+t_0)\ge 10$.
The above choice ensures that
\begin{equation}
\label{C1g} 
\forall t\ge t_0\colon\quad
c_2 \ge \delta\, g(t)^2\quad\text{and}\quad 0<g(t)\le 1
\end{equation}
that will be useful later on.

Following Schonbek's Fourier splitting method, see \cite{Sch85}, we introduce
\begin{equation*}
I_{\uep}(t)=\int_{|\xi|\le g(t)} |\widehat \uep(\xi,t)|^2\dd \xi,
\qquad
I_{\wep}(t)=\int_{|\xi|\le g(t)} |\widehat \wep(\xi,t)|^2\dd \xi,
\qquad
I_{\zep}(t)=\int_{|\xi|\le g(t)} |\widehat \zep(\xi,t)|^2\dd \xi.
\end{equation*}
We have
\begin{multline}\label{fin*}
\|\nabla \uep(t)\|_{L^2}^2
=\frac{1}{(2\pi)^3}\int |\xi|^2|\widehat\uep(\xi,t)|^2\dd\xi
\ge \frac{g(t)^2}{(2\pi)^3}\int_{|\xi|\ge g(t)}|\widehat \uep(\xi,t)|^2\dd\xi\\
=g(t)^2\|\uep(t)\|_{L^2}^2 - \frac{g(t)^2}{(2\pi)^3} I_{\uep}(t).
\end{multline}
By~\eqref{EI}, for all $t\geq s\ge t_0$,
\begin{equation*}
\begin{split}
\|\uep(t)\|_{L^2}^2+\|\wep(t)\|_{L^2}^2 
    &+\int_s^t 
    \Bigl[c_1 g(r)^2\|\uep(r)\|_{L^2}^2
      + c_2\|\wep(r)\|_{L^2}^2
    \Bigr]
    \dd r \\
    &\le 
    \int_s^t 
    \frac{ c_1 g(r)^2}{(2\pi)^3}  I_{\uep}(r)\dd r
    +\|\uep(s)\|_{L^2}^2+\|\wep(s)\|_{L^2}^2.
\end{split}
\end{equation*}
Using \eqref{C1g} we get, for all $t\geq s\ge t_0$,
\begin{equation}
\label{befow}
\|\zep(t)\|_{L^2}^2 
     +\int_s^t
    \delta g(r)^2 \|\zep(r)\|_{L^2}^2\dd r 
    \le 
    C\int_s^t  g(r)^2I_{\zep}(r)\dd r     +\|\zep(s)\|_{L^2}^2,
\end{equation}

Let 
\[
e(t)=\exp\Bigl(\int_{t_0}^t \delta\, g(s)^2\dd s\Bigr)
\]
so that $e(t_0)=1$. Applying the singular Gronwall-type lemma in \cite[Lemma~3.6]{Brando-Revista},
we deduce
\begin{equation}
\label{wis1}
\begin{split}
e(t)\|\zep(t)\|_{L^2}^2
&\le \|\zep(t_0)\|_{L^2}^2
    	+C\int_{t_0}^t e(s)g(s)^2I_{\zep}(s)\dd s.
\end{split}    	    	  	
\end{equation}
(Notice that~\eqref{wis1} would be obvious if the map $t\mapsto \|\zep(t)\|_{L^2}^2$ were differentiable
as in this case~\eqref{befow} leads to differential inequality for $\|\zep(t)\|_{L^2}^2$. The above mentioned
singular Gronwall-type lemma provides a rigorous justification).

\paragraph{Step 2. Low frequency estimates.}

We recall the notation introduced in \eqref{notationsgr}. Applying the Duhamel formula in \eqref{MMP} and taking the Fourier transform  we get
\begin{equation*}
\widehat \zep(\xi,t)
 = \widehat{\zl}(\xi,t)
 	+\int_0^t K(\xi,t-s)\mathcal{F}
  	\begin{pmatrix}
	-\P\dive ((\jep\uep)\otimes \uep)\\
	-\dive ((\jep\uep)\otimes \wep)
	\end{pmatrix}
	(\xi,s)\dd s.	
\end{equation*}
But $|\mathcal{F}((\jep \uep)\otimes \uep)|(\xi,t)\le \|\uep(t)\|_{L^2}^2\le \|\zep(t)\|_{L^2}^2$
and $|\mathcal{F}((\jep\uep)\otimes \wep)|(\xi,t)\le \|\uep(t)\|_{L^2}\,\|\wep\|_{L^2}\le \|\zep(t)\|_{L^2}^2$.
Moreover, the operators $\P\dive $ and $\dive $ are Fourier multipliers with symbol
bounded by $C|\xi|$.
By the first property in~\eqref{infoK}, all the terms in previous formula are well defined.
Hence,
\begin{equation*}
\begin{split}
|\widehat \zep(\xi,t)|
 &\lesssim |\widehat{\zl}(\xi,t)|
 	+C|\xi| \int_0^t \|\zep(s)\|_{L^2}^2\dd s
\end{split}
\end{equation*}
Integrating $|\widehat \zep(\xi,t)|^2$ in the ball $\{|\xi|\le g(t)\}$, we get
\begin{equation}
\label{iteration0c}
I_{\zep}(t)\le C\|e^{tM}z_0\|_{L^2}^2+Cg(t)^{5}\Bigl(\int_0^t \|\zep(s)\|_{L^2}^2\dd s\Bigr)^2.
\end{equation}

Then we deduce from~\eqref{wis1}
\begin{multline}
\label{wis2}
e(t)\|\zep(t)\|_{L^2}^2
\le \|\zep(t_0)\|_{L^2}^2
    	+C\int_{t_0}^t e(s)g(s)^2 \|e^{sM}z_0\|_{L^2}^2\dd s\\
	+C\int_{t_0}^t e(s)g(s)^7\Bigl(\int_0^s\|\zep(\tau)\|_{L^2}^2\dd\tau\Bigr)^2\dd s.
\end{multline}

But from the energy inequality~\eqref{EI}, for any $t$,
$
\|\zep(t)\|_{L^2}^2\le \|z_0\|_{L^2}^2.
$
Then we deduce from the above relation
\begin{equation*}
e(t)\|\zep(t)\|_{L^2}^2\le \|z_0\|_{L^2}^2+C\int_{t_0}^t e(s)g(s)^2\|e^{sM}z_0\|_{L^2}^2\dd s
+C\|z_0\|_{L^2}^4\int_{t_0}^te(s)s^2g(s)^7\dd s.
\end{equation*}
Recalling that $g(t)^2=\textstyle\frac{10}{\delta}(1+t)^{-1}$, and the definition of $e$, we have
\begin{equation}\label{defet}
 e(t)=\displaystyle\frac{(1+t)^{10}}{(1+t_0)^{10}}.
\end{equation}
Then, for all $t\ge t_0$,
\begin{equation*}
\begin{split}
(1+t)^{10}\|\zep(t)\|_{L^2}^2
&\le C\|z_0\|_{L^2}^2+C\int_{t_0}^t (1+s)^{9}\|e^{sM}z_0\|_{L^2}^2\dd s
  +C\|z_0\|_{L^2}^4\int_{t_0}^t (1+s)^{12-7/2}.\\
\end{split}
\end{equation*}
We infer that
\begin{equation}
\label{nodeca}
\|\zep(t)\|_{L^2}^2\le C\Psi(t),
\end{equation}
where
\begin{equation*}
\Psi(t):=\frac{1}{1+t}\int_{t_0}^t \|e^{sM}z_0\|_{L^2}^2\dd s+(\|z_0\|_{L^2}^2+\|z_0\|_{L^2}^4)(1+t)^{-1/2}.
\end{equation*}
We recall \eqref{notationsgr} and the properties of the kernel $K$ stated in \eqref{infoK}. The dominated convergence theorem implies that 
$\|e^{tM}z_0\|_{L^2}^2\to0$ as $t\to+\infty$. We deduce that $\Psi(t)\to0$ as $t\to+\infty$, and so
\[
\|\zep(t)\|_{L^2}^2\to0 \qquad\text{as $t\to+\infty$}
\]
uniformly in $\varepsilon$.

Let us now prove Assertion ii).

\paragraph{Step 3. Algebraic decay.}

Let $\Gamma>0$ and $C_\Gamma>0$ be such that $\|e^{tM}z_0\|_{L^2}^2\le C_\Gamma(1+t)^{-\Gamma}$ for all~$t\ge0$.  From~\eqref{nodeca} we can assert that
$\|\zep(t)\|_{L^2}^2\le C\Psi(t)\le C(1+t)^{-\min\{\Gamma,1/2\}}$.

Therefore, Assertion ii) is already established in the case $0<\Gamma\le1/2$. Let now  $1/2<\Gamma\le 5/2$.
In this case we can go back to~\eqref{wis2} and use therein the estimate $\|\zep(t)\|_{L^2}^2\le C(1+t)^{-1/2}$ obtained at the beginning of the present step.
In particular, $\bigl(\int_0^s\|\zep(\tau)\|_{L^2}^2\dd\tau\bigr)^2\le C(1+s)$, and we can now use this in \eqref{wis2} to obtain
\begin{equation*}
\begin{split}
(1+t)^{10}\|\zep(t)\|_{L^2}^2
&\le C+C\int_{t_0}^t (1+s)^{9}\|e^{sM}z_0\|_{L^2}^2\dd s
  +C\int_{t_0}^t (1+s)^{11-7/2}.\\
\end{split}
\end{equation*}

This allows to improve~\eqref{nodeca}, 
by an estimate of the form 
\[
\|\zep(t)\|_{L^2}^2\le C(1+t)^{-\Gamma} +C(1+t)^{-3/2},
\]
which is enough to establish Assertion~ii) in the case $1/2<\Gamma\le 3/2$.
After one more iteration of the same argument we obtain
\[
\|\zep(t)\|_{L^2}^2\le C(1+t)^{-\Gamma}+C(1+t)^{-5/2},
\]
which can no longer be improved.
Assertion~ii) is completely proved.

We prove now Assertion~iii). 

\paragraph{Step 4. Estimate for $\|\zep-\zl\|_{L^2}^2$.}
First of all, the previously obtained decay on $\|\zep(t)\|_{L^2}^2$
implies
\[
\|\zep(t)\|^2\le (1+t)^{-\Gamma}\zeta_\Gamma(t),
\]
where $\zeta_\Gamma$ is independent on $\varepsilon$, and $\zeta_\Gamma(t)=O(1)$ 
if $0<\Gamma\le 5/2$ and $\zeta_\Gamma(t)=o(1)$ if $\Gamma=0$, as $t\to+\infty$.

We rely on the energy estimate~\eqref{EID2}.
Recalling that $0<\delta\le c_3$, we can write, for all $t\ge s\ge t_0$,
\begin{equation*}
\begin{split}
 \|(\zep-\zl)(t)\|_{L^2}^2 - \|(\zep-\zl)(s)\|_{L^2}^2
 + \delta\int_s^t\Bigl( \|\nabla &(\uep-\ul)\|_{L^2}^2+\|(\wep-\wl)\|_{L^2}^2\Bigr)\\
 &\le c_4\int_s^t \|\uep\|_{L^2}^2\Bigl(\|\ul\|_{L^\infty}^2+\|\nabla \wl\|_{L^\infty}^2\Bigr)\\
 & \le C \int_s^t \zeta_\Gamma(r)\,r^{-3/2-2\Gamma}\dd r. 
 \end{split}
\end{equation*}
Here, in the last inequality we used~\eqref{AA} and the previously obtained decay on $\|\uep(t)\|_{L^2}^2$.

Let us denote
\[
I_{\uep-\ul}(t):=\int_{|\xi|\le g(t)} |\widehat{\uep}-\widehat{\ul}\,|^2(\xi,t)\dd\xi,
\qquad\text{and}\qquad
I_{\zep-\zl}(t):=\int_{|\xi|\le g(t)} |\widehat{\zep}-\widehat{\zl}\,|^2(\xi,t)\dd\xi.
\]
By Fourier splitting, see \eqref{fin*},
\[
\|\nabla(\uep-\ul)(t)\|_{L^2}^2\ge g(t)^2 \|\uep-\ul\|_{L^2}^2-\frac{g(t)^2}{(2\pi)^3}I_{\uep-\ul}(t),
\]
so, for all $t\ge s\ge t_0$,
\begin{align}
\|(\zep-\zl)(t)\|_{L^2}^2-\|(\zep-\zl)(s)\|_{L^2}^2
+\delta\int_s^t&\Bigl( g(r)^2\|(\uep-\ul)(r)\|_{L^2}^2+\|(\wep-\wl)(r)\|_{L^2}^2\Bigr)\dd r\notag\\
&\le C  \int_s^t \Bigl(g(r)^2I_{\uep-\ul}(r)+\zeta_\Gamma(r)\,r^{-3/2-2\Gamma}\Bigr)\dd r\label{zepldif}\\
&\le C  \int_s^t \Bigl(g(r)^2I_{\zep-\zl}(r)+\zeta_\Gamma(r)\,r^{-3/2-2\Gamma}\Bigr)\dd r.\notag
\end{align}
We bound now $I_{\zep-\zl}$. From the Duhamel formula, since $z-\zl$ vanishes at $t=0$, we have
\[
(\widehat \zep-\widehat{\zl}\,)(t)=
 	\int_0^t \widehat K(\xi,t-s)\mathcal{F}
  	\begin{pmatrix}
	-\P\dive ((\jep\uep)\otimes \uep)\\
	-\dive ((\jep \uep)\otimes \wep)
	\end{pmatrix}
	(\xi,s)\dd s.	
\]
Hence, as in~\eqref{iteration0c}, using the $L^2$-decay of $\zep$ established in the previous steps, we obtain
\begin{equation}
\label{low-d}
I_{\zep-\zl}(t)
\le C g(t)^5\Bigl(\int_0^t \|\zep(s)\|_{L^2}^2\dd s\Bigr)^2
\le Cg(t)^5\Bigl(\int_0^t \zeta_\Gamma(s)\,(1+s)^{-\Gamma}\dd s\Bigr)^2.
\end{equation}

We recall that for all $t\ge t_0$,  $g(t)^2\le1$, see \eqref{C1g}. From \eqref{zepldif} and \eqref{low-d} we obtain, for all $t\geq s\ge t_0$,
\begin{equation*}
\begin{split}
\|(\zep-\zl)(t)\|_{L^2}^2- \|(\zep-\zl)(s)\|_{L^2}^2&
 +\int_s^t \delta g(r)^2\|(\zep-\zl)(r)\|_{L^2}^2\dd r\\
&\le C \int_s^t \Bigl( g(r)^2I_{\zep-\zl}(r)+\zeta_\Gamma(r)\,r^{-3/2-2\Gamma}\Bigr)\dd r\\
&\le C \int_s^t \Bigl(g(r)^7\Bigl(\int_0^t \zeta_\Gamma(r)\,(1+r)^{-\Gamma}\dd r\Bigr)^2
+\zeta_\Gamma(r)\,r^{-3/2-2\Gamma}\Bigr)\dd r.
\end{split}
\end{equation*}

Let $e(t)$ as in Step~1.
Applying  the singular Gronwall-type lemma from \cite[Lemma~3.6]{Brando-Revista} in the inequality above, we obtain
\begin{multline*}
e(t)\|(\zep-\zl)(t)\|_{L^2}^2
\le \|(\zep-\zl)(t_0)\|_{L^2}^2 +C\int_{t_0}^t e(s)g(s)^7\Bigl(\int_0^s \zeta_\Gamma(r)\,(1+r)^{-\Gamma}\dd r\Bigr)^2\dd s\\
+C\int_{t_0}^t e(s)\zeta_\Gamma(s)\,s^{-3/2-2\Gamma}\dd s.
\end{multline*}

 Recalling the expression of  $g$, see~\eqref{choi-g}, and the definition of the function~$e$, see \eqref{defet},
we obtain
\begin{multline*}
(1+t)^{10}\|(\zep-\zl)(t)\|_{L^2}^2
\le C 
+C\int_{t_0}^t (1+s)^{10-7/2}\Bigr(\int_0^s\zeta_\Gamma(\tau)\,(1+\tau)^{-\Gamma}\dd \tau\Bigr)^2\dd s\\
+ C\int_{t_0}^t\zeta_\Gamma(s)\,(1+s)^{10-3/2-2\Gamma}\dd s.
\end{multline*}
Hence, recalling that $\zeta_\Gamma(t)$ is a bounded function, we obtain
\[
\|(\zep-\zl)(t)\|_{L^2}^2\le C
\begin{cases}
(1+t)^{-1/2-2\Gamma}& \text{if $0\le\Gamma<1$}\\
(1+t)^{-5/2}(\log(e+t))^2 &\text{if $\Gamma=1$}\\
(1+t)^{-5/2}&\text{if $1<\Gamma\le 5/2$}.
\end{cases}
\]
When $\Gamma=0$, we can use the stronger property that $\zeta_\Gamma(t)=o(1)$ as $t\to+\infty$.
Then, in this case, we can slightly improve the decay and obtain, uniformly with respect to $\varepsilon$,
\[
\|(\zep-\zl)(t)\|_{L^2}^2\le (1+t)^{-1/2}o(1),\qquad\text{as $t\to+\infty$}.
\]
\end{proof}

\begin{proof}[Proof of Theorem~\ref{th:decay}]
The first assertion of Theorem~\ref{th:decay} follows from Proposition~\ref{prop:decay} and from the assertion e) of Proposition~\ref{prop:exist}.
Let us prove the second and third assertion of Theorem~\ref{th:decay}.
By our assumption  
$\|\ul(t)\|_{L^2}^2=O(t^{-\Gamma})$ as $t\to+\infty$, 
for some $\Gamma\in [0,5/2]$. 
We claim that this implies
\begin{equation*}
\|\wl(t)\|_{L^2}^2=O(t^{-\Gamma-1}),
\qquad\text{as $t\to+\infty$}.
\end{equation*}
This crucial property of the linear micropolar system will be established
in the appendix (see Proposition~\ref{prop:impro} below), as a consequence of a remarkable enstrophy identity.
In particular, it follows that $\|\zl(t)\|_{L^2}^2=O(t^{-\Gamma})$ as $t\to+\infty$.
Now the two last assertions of Theorem~\ref{th:decay} directly follow applying Proposition~\ref{prop:decay} and assertions e) and f) of Proposition~\ref{prop:exist}.
\end{proof}

\section{Further remarks}

We discuss here some consequences of Theorem~\ref{th:decay}. An important application of Theorem~\ref{th:decay} is the improvement of the decay of $\|w(t)\|_{L^2}^2$.
The first assertion of the following corollary is interesting in the case $\Gamma=0$, i.e., when there is
no additional restriction on the data other than the usual $L^2$-condition.

\begin{corollary}[Stronger decay for $w$]
\label{cor:decay}
Let $\mu,\chi>0$ and $\gamma,\kappa\ge0$.
Let $(u,w)$ be a restricted Leray solution
of~\eqref{MP}, with initial data $z_0=(u_0,w_0)\in L^2_\sigma(\R^3)\times L^2(\R^3)$.
\begin{itemize}
\item[i)] The decay of $\|w(t)\|_{L^2}^2$ obtained in Theorem~\ref{th:decay}, i) can be improved as follows:
	\begin{itemize}
	\item If $\gamma>0$, then $t\|w(t)\|_{L^2}^2\to0$ as $t\to+\infty$.
	\item If $\gamma=0$ and  $\int (1+|\xi|)|\widehat z_0(\xi)|\dd\xi <\infty $, 
	then $t^{1/2}\|w(t)\|_{L^2}^2\to0$ as $t\to+\infty$.
       \end{itemize}
\item[ii)]  Let $0<\Gamma\le 5/2$  and $\|\ul(t)\|_{L^2}^2=O(t^{-\Gamma})$ as $t\to+\infty$.
	\begin{itemize}
	\item If $\gamma>0$, 	then $\|w(t)\|_{L^2}^2=O(t^{-1-\Gamma})$ as $t\to+\infty$.
	\item If $\gamma=0$  and  $\int (1+|\xi|)|\widehat z_0(\xi)|\dd\xi <\infty $, 
		then,  as $t\to+\infty$,
		\[\|w(t)\|_{L^2}^2=
			\begin{cases} 
			O(t^{-1/2-2\Gamma})&\text{if $0<\Gamma\le1/2$}\\
			O(t^{-1-\Gamma})&\text{if $1/2\le \Gamma\le3/2$}\\
			O(t^{-5/2})&\text{if $3/2\le \Gamma\le5/2$}.
			\end{cases} 
			\]
       \end{itemize}
\end{itemize}
\end{corollary}

\begin{proof}[Proof of Corollary~\ref{cor:decay}]
The case $\gamma>0$ and $\Gamma=0$ is already known, see \cite{Guterres2} and \cite{Guterres3}.
For the case $\gamma>0$ and $0\le \Gamma\le 5/2$, we observe that with our assumptions, by Theorem~\ref{th:decay}, $\|u(t)\|_{L^2}^2=O(t^{-\Gamma})$. But, as proved in~\cite{Guterres2} and \cite{Guterres3},
the implication $\|u(t)\|_{L^2}^2=O(t^{-\Gamma})\;\Rightarrow \|w(t)\|_{L^2}^2=O(t^{-1-\Gamma})$ holds true, so the assertion follows in this case.

It remains to study the case $\gamma=0$.
We start observing that, for any $\Gamma\ge0$, applying Proposition \ref{prop:impro}, we obtain from the assumption
$\|\ul(t)\|_{L^2}^2=O(t^{-\Gamma})$, that
$\|\wl(t)\|_{L^2}^2=O(t^{-1-\Gamma})$.
But $\|w(t)\|_{L^2}^2\le 2\|\wl(t)\|_{L^2}^2+2\|(w-\wl)(t)\|_{L^2}^2$, so the conclusion follows applying Part~iii) of Theorem~\ref{th:decay}.
 \end{proof}

As an illustration,
let us consider the following particular case.
We assume as usual that $\int (1+|\xi|)|\widehat z_0(\xi)|\dd\xi <\infty$ when $\gamma=0$.
Let $1\le p\le 3/2$ and  
let $u_0\in (L^p\cap L^2_\sigma)(\R^3)$. For $w_0$, we put no additional integrability condition, besides $w_0\in L^2(\R^3)$.
From the usual $L^p$-$L^2$ heat kernel estimates we infer, as $t\to+\infty$, $\|e^{\mu t\Delta}u_0\|_{L^2}=O(t^{-(3/2)(1/p-1/2)})$
and $\|\curl e^{\mu t\Delta}w_0\|_{L^2}=O(t^{-1/2})$.
When $6/5\le p\le 3/2$, the latter decay is faster (or equal) than the previous one and
therefore, by Theorem~\ref{th:linear-pro}, $\|\ul(t)\|_{L^2}=O(t^{-(3/2)(1/p-1/2)})$, as $t\to+\infty$.
Corollary~\ref{cor:decay} applies in the range $1/2\le\Gamma=3(1/p-1/2)\le 3/2$ and it follows that, for $6/5\le p\le 3/2$,
as $t\to+\infty$,
\[
\|u(t)\|_{L^2}^2=O(t^{-3(1/p-1/2)})
\quad\text{and}\quad
\|w(t)\|_{L^2}^2=O(t^{-1-3(1/p-1/2)}).
\]
When $1\le p<6/5$, to reach the same conclusion as above, we need
the better integrability condition $w_0\in (L^q\cap L^2)(\R^3)$, with $\frac1q=\frac1p-\frac13$. In particular, for $p=1$, we deduce that, when  
$(u_0,w_0)\in (L^1\cap L^2_\sigma)\times (L^{3/2}\cap L^2)(\R^3)$, then
as $t\to+\infty$,
\[
\|u(t)\|_{L^2}^2=O(t^{-3/2})
\quad\text{and}\quad
\|w(t)\|_{L^2}^2=O(t^{-5/2}).
\]

\section*{Appendix}
\addcontentsline{toc}{section}{Appendix}
\renewcommand{\theequation}{A.\arabic{equation}}
\renewcommand{\thepropappendix}{A.\arabic{propappendix}}

The aim of this appendix is twofold. First, we will prove an enstrophy identity that holds true for solutions of the linear system \eqref{LMP}, see Proposition \ref{prop:enstro} below. This is similar to the 2D enstrophy identity proved in \cite[Lemma 3.1]{brandolese_2d_2024}, except that, due to a crucial cancellation of the nonlinear terms valid in 2D only, the 2D enstrophy identity holds true for the nonlinear system too and not only for the linear system. Next, we show that the enstrophy identity implies the following property: a rate of decay as $t\to\infty$ for the energy of $\ul$ implies a better rate of decay for the energy of $\wl$ and $\nabla\ul$. This implication is the key property that allowed us to simplify the hypothesis of Theorem \ref{th:decay}, where no decay condition on $\|\wl\|_{L^2}$ was required. These two properties are new and important. We put it in a appendix because their proofs are essentially the same as in dimension two, see  Lemma 3.1, Corollary 3.4, Proposition 4.3 and Remark 4.4 from \cite{brandolese_2d_2024}.

We use the notation from Section \ref{linearsect}. We define in addition the following quantity:
\begin{equation*}
\el  :=\hl-\frac12\Omegal
\end{equation*}
and we recall that $\Omegal = \curl  \ul$ and $\hl=\P\wl$. Let us also introduce the following constant:
\begin{equation}
\label{choice-a}
 a:=\frac{\gamma\chi+\mu\chi+2\mu\gamma}{4\chi^2}.
\end{equation}

\begin{propappendix}
\label{prop:enstro}
Let $(u_0,w_0)\in L^2_\sigma(\R^3)\times L^2(\R^3)$
and $\chi,\mu>0$, $\gamma\ge0$.
The following enstrophy identity holds:
\begin{multline}
 \label{enEO2}
 \frac12\frac{\dd }{\dd t}\Bigl(\|\el \|_{L^2}^2+a\|\Omegal\|_{L^2}^2\Bigr)
 +4\chi\|\el \|_{L^2}^2
 +(\gamma+\chi)\bigl\|\nabla\el -\frac{\mu}{2\chi}\nabla\Omegal\bigr\|_{L^2}^2\\
 +\frac{\mu\gamma}{4\chi^2}(\mu+\chi)\|\nabla\Omegal\|_{L^2}^2=0.
\end{multline}
Moreover,
\[
\|\el(t)\|_{L^2}^2=o(t^{-1}),\qquad \|\Omegal(t)\|_{L^2}^2=o(t^{-1}),\qquad
\| \hl(t)\|_{L^2}^2=o(t^{-1}),\qquad
\text{as $t\to+\infty$.}
\]
\end{propappendix}

\begin{proof}

We have from~\eqref{sysL}
\begin{equation*}
 \begin{cases}
 \dert \Omegal=(\mu+\chi)\Delta \Omegal-2\chi\Delta \hl\\
 \dert \hl=\gamma\Delta \hl+2\chi\Omegal-4\chi \hl.
 \end{cases}
\end{equation*}
Eliminating $\hl=\el +\frac12\Omegal$ from \eqref{LVMP}, we obtain the following new system
\begin{equation*}
\begin{cases}
\dert\Omegal=\mu\Delta\Omegal-2\chi\Delta\el \\ 
 \dert\el =(\gamma+\chi)\Delta\el +\frac{\gamma-\mu}{2}\Delta\Omegal-4\chi\el .
\end{cases}
\end{equation*}
Multiplying the first equation by $\Omegal$ and the second equation by $\el $, we obtain, after a few integrations by parts
\[
\begin{cases}
\frac12\frac{\dd }{\dd t}\|\Omegal\|_{L^2}^2+\mu\int|\nabla \Omegal|^2 -2\chi\int\nabla\Omegal\cdot\nabla\el =0\\
\frac12\frac{\dd }{\dd t}\|\el \|_{L^2}^2+(\gamma+\chi)\int|\nabla \el |^2
+\frac{\gamma-\mu}{2} \int\nabla\Omegal\cdot\nabla\el +4\chi\int|\el |^2=0.
\end{cases}
\]
Multiplying by $a$ the first identity and adding to the second identity
we obtain
\begin{multline}\label{a3bis}
\frac12\frac{\dd }{\dd t}\Bigl(\|\el \|_{L^2}^2+a\|\Omegal\|_{L^2}^2\Bigr)
 +4\chi\int|\el |^2
 + (\gamma+\chi)\int|\nabla\el |^2\\
 +\Bigl(\frac{\gamma-\mu}{2}-2\chi a\Bigr)\int\nabla\el \cdot\nabla\Omegal
 +a\mu\int|\nabla\Omegal|^2=0.
\end{multline}
The non-negativity property
\[
(\gamma+\chi)|\nabla\el |^2
 +\Bigl(\frac{\gamma-\mu}{2}-2\chi a\Bigr)\nabla\el \cdot\nabla\Omegal
 +a\mu|\nabla\Omegal|^2 \ge0
\]
for all $\nabla\el \in \R^{3\times3}$ and all  $\nabla\Omegal\in \R^{3\times3}$,
is guaranteed by the following condition
\begin{equation*}
 \Bigl(\frac{\gamma-\mu}{2}-2\chi a\Bigr)^2-4(\gamma+\chi)a\mu\le0,
\end{equation*}
that can be rewritten  as
\begin{equation*}
 16\chi^2a^2-8(\gamma\chi+\mu\chi+2\mu\gamma)a+(\gamma-\mu)^2\le0.
\end{equation*}
The choice of $a$ given in \eqref{choice-a} is indeed suitable, and corresponds to the minimizer of the above expression. When $\mu\gamma=0$, this is in fact the only possible choice. With this choice of $a$, applying the square reduction yields
\[
\begin{split}
  (\gamma+\chi)\int|\nabla\el |^2&
 +\Bigl(\frac{\gamma-\mu}{2}-2\chi a\Bigr)\int\nabla\el \cdot\nabla\Omegal
 +a\mu\int|\nabla\Omegal|^2\\
 &=(\gamma+\chi)\int|\nabla\el |^2-\frac{\mu(\chi+\gamma)}{\chi}\int\nabla\el \cdot\nabla\Omegal
 +\frac{\mu}{4\chi^2}(\gamma\chi+\mu\chi+2\mu\gamma)\int|\nabla\Omegal|^2\\
 &=(\gamma+\chi)\int\Bigl|\nabla\el -\frac{\mu}{2\chi}\nabla\Omegal\Bigr|^2
 +\frac{\mu\gamma}{4\chi^2}(\mu+\chi)\int|\nabla\Omegal|^2.
 \end{split}
\]
Using this in \eqref{a3bis}  implies identity~\eqref{enEO2}.

The same argument as in the proof of \eqref{E1} allows to prove that the solution to the linear system~\eqref{sysL} satisfies the energy inequality, for all $t\ge0$,
\begin{multline*}
\|\ul(t)\|_{L^2}^2+\|\hl(t)\|_{L^2}^2  +2(\mu+\chi)\int_0^t\|\nabla \ul(s)\|_{L^2}^2\dd s+2\gamma\int_s^t\|\nabla \hl(s)\|_{L^2}^2\dd s
+8\chi\int_0^t\|\hl(s)\|_{L^2}^2\dd s\\
\le   \|u_0\|_{L^2}^2+\|w_0\|_{L^2}^2 +8\chi\int_0^t\|\hl(s)\|_{L^2}\,\|\nabla \ul(s)\|_{L^2}\dd s.
\end{multline*}
But, for any $\eta>0$, we have 
$2\|\hl\|_{L^2}\,\|\nabla \ul\|_{L^2}
\le \frac{1}{\eta}\|\hl\|_{L^2}^2+\eta\|\nabla \ul\|_{L^2}^2$.
Choosing  $\frac12< \eta<\frac{\mu+\chi}{2\chi}$, for example,
\[
\eta=\textstyle\frac14(1+\frac{\mu+\chi}{\chi}),
\]
implies
\begin{multline*}
\|\ul(t)\|_{L^2}^2+\|\hl(t)\|_{L^2}^2
+c_1\int_0^t \|\nabla \ul(s)\|_{L^2}^2\dd s
	+2\gamma\int_0^t \|\nabla \hl(s)\|_{L^2}^2\dd s
+c_2\int_0^t\|\hl(s)\|_{L^2}^2\dd s\\
\le \|u_0\|_{L^2}^2+\|w_0\|_{L^2}^2,
\end{multline*}
with $c_1,c_2>0$ given by
\begin{equation*}
c_1=2(\mu+\chi-2\chi\eta)\quad\text{and}\quad
c_2=2\chi(4-2/\eta).
\end{equation*}
We infer that
\[
\int_0^\infty \big(c_1\|\nabla\ul\|_{L^2}^2
+c_2\|\hl\|_{L^2}^2\bigr)\le\|u_0\|_{L^2}^2+\|w_0\|_{L^2}^2.
\]
Then, recalling the definition of $\el $, we deduce
\begin{equation}\label{elint}
\int_0^\infty \bigl(\|\el \|_{L^2}^2+a\|\Omegal\|_{L^2}^2\bigr)
\lesssim \|u_0\|_{L^2}^2+\|w_0\|_{L^2}^2.
\end{equation}
Now, the enstrophy identity~\eqref{enEO2} implies that
$\frac{\dd }{\dd t}\bigl(\|\el \|_{L^2}^2+a\|\Omegal\|_{L^2}^2\bigr)\le0$,
so the function $t\mapsto \|\el \|_{L^2}^2+a\|\Omegal\|_{L^2}^2$ is monotonically non-increasing.
Then, for all $t>0$, we can bound from below
\[
 \int_{t/2}^{t} \bigl(\|\el \|_{L^2}^2+a\|\Omegal\|_{L^2}^2\bigr)
\geq \frac t2\bigl(\|\el (t)\|_{L^2}^2+a\|\Omegal(t)\|_{L^2}^2\bigr).
\]
Since, by \eqref{elint}, the function $t\mapsto \|\el(t) \|_{L^2}^2+a\|\Omegal(t)\|_{L^2}^2$ belongs to $L^1(\R_+)$, the lhs of the relation above goes to 0 as $t\to\infty$. Then $t\bigl(\|\el (t)\|_{L^2}^2+a\|\Omegal(t)\|_{L^2}^2\bigr)\to0$ as
$t\to+\infty$, so it holds also $t\|\hl(t)\|_{L^2}^2\to0$ as
$t\to+\infty$.
\end{proof}

As an application of this enstrophy estimate, we illustrate how
this estimate can be used to improve the decay of the linear micro-rotation.

\begin{propappendix}\label{prop:impro}
Let $(u_0,w_0)\in L^2_\sigma(\R^3) \times L^2(\R^3)$ and let $\Gamma\ge0$.
If the solution $(\ul,\wl)$ of the linear problem~\eqref{LMP} satisfies $\|\ul(t)\|_{L^2}^2=O(t^{-\Gamma})$ as $t\to+\infty$, then
\begin{equation}\label{a3}
\|\wl(t)\|_{L^2}^2+\|\nabla\ul(t)\|_{L^2}^2=O(t^{-\Gamma-1}),
\qquad\text{as $t\to+\infty$}.
\end{equation}
\end{propappendix}

\begin{proof}
We observe first that \eqref{a3} is equivalent to the following relation:
\begin{equation}
\label{a4}
\|\hl(t)\|_{L^2}^2+\|\nabla\ul(t)\|_{L^2}^2=O(t^{-\Gamma-1}),
\qquad\text{as $t\to+\infty$}.
\end{equation}
Indeed, we have from \eqref{soldivw} that $\hl-\wl$ decays exponentially fast in $L^2$ as $t\to\infty$.

We will show that \eqref{a4} holds true. We start by proving a weaker assertion, namely that
\begin{equation}
\label{impro1}
\forall \,\Gamma\ge0,\quad
\|\ul(t)\|_{L^2}^2+\|\hl(t)\|_{L^2}^2 =O(t^{-\Gamma})
\;\Longrightarrow\;
\begin{cases}
\|\hl(t)\|_{L^2}^2 =O(t^{-\Gamma-1}),\\
\|\nabla \ul(t)\|_{L^2}^2=O(t^{-\Gamma-1})
\end{cases}
\end{equation}
as $t\to+\infty$.

Let $\alpha>\Gamma$.
Multiplying the enstrophy identity~\eqref{enEO2} by $2t^{\alpha+1}$ and integrating by parts in time
gives
\begin{multline*}
 t^{\alpha+1}\Bigl(\|\el \|_{L^2}^2+a\|\Omegal\|_{L^2}^2\Bigr)(t)
 -(\alpha+1)\int_0^t s^\alpha\Bigl(\|\el \|_{L^2}^2+a\|\Omegal\|_{L^2}^2\Bigr)(s)\dd s\\
 +2\int_0^t s^{\alpha+1}\biggl[ 4\chi\|\el \|_{L^2}^2
 +(\gamma+\chi)\int\Bigl|\nabla\el -\frac{\mu}{2\chi}\nabla\Omegal\Bigr|^2
 +\frac{\mu\gamma}{4\chi^2}(\mu+\chi)\int|\nabla\Omegal|^2\biggr]\dd s=0.
\end{multline*}
Hence, for all $t\ge0$.
\begin{equation}
\label{porau}
 t^{\alpha+1}\Bigl(\|\el \|_{L^2}^2
 +a\|\Omegal\|_{L^2}^2\Bigr)(t)
 \le (\alpha+1)\int_0^t s^\alpha\Bigl(\|\el \|_{L^2}^2+a\|\Omegal\|_{L^2}^2\Bigr)(s)\dd s.
\end{equation}
Similarly, from~\eqref{sysL}, multiplying the first and the second equation of the linear micropolar  system by $t^\alpha \ul$ and $t^\alpha\hl$ respectively,
and integrating in space and by parts in time, gives the following variant
of the energy inequality~\eqref{EI}:
 \begin{multline*}
\label{EI}
t^\alpha\Bigl(\|\ul(t)\|_{L^2}^2+\|\hl(t)\|_{L^2}^2\Bigr)
+\int_0^t s^\alpha\biggl[ c_1\|\nabla \ul(s)\|_{L^2}^2
+c_2\|\hl(s)\|_{L^2}^2\biggr]\dd s\\
\le \alpha\int_0^t s^{\alpha-1}\Bigl(\|\ul(s)\|_{L^2}^2+\|\hl(s)\|_{L^2}^2\Bigr)\dd s,
\end{multline*}
Therefore, the condition $\|\ul(t)\|_{L^2}^2+\|\hl(t)\|_{L^2}^2 =O(t^{-\Gamma})$
implies
\[
\int_0^t s^\alpha\Bigl( \|\nabla \ul(s)\|_{L^2}^2
	+\|\hl(s)\|_{L^2}^2\Bigr)\dd s=O(t^{\alpha-\Gamma})
	\qquad\text{as $t\to+\infty$}.
\]
But, by the definition of $\el $ and the triangular inequality,
we have $s^\alpha(\|\el \|_{L^2}^2+a\|\Omegal\|_{L^2}^2)
\lesssim s^\alpha(\|\nabla\ul\|_{L^2}^2+\|\hl\|_{L^2}^2)$,
hence, from~\eqref{porau},
\begin{equation*}
 t^{\alpha+1}\Bigl(\|\el \|_{L^2}^2+a\|\Omegal\|_{L^2}^2\Bigr)(t)
 = O(t^{\alpha-\Gamma}).
\end{equation*}
But, since $\|\Omegal\|_{L^2}=\|\nabla\ul\|_{L^2}$, by the triangular inequality we deduce
$\|\hl\|_{L^2}^2
+\|\nabla \ul\|_{L^2}^2\lesssim \|\el \|_{L^2}^2+a\|\Omegal\|_{L^2}^2$ and finally
\[
\|\hl\|_{L^2}^2+\|\nabla \ul\|_{L^2}^2=O(t^{-\Gamma-1}).
\]
This establishes the validity of the implication~\eqref{impro1}.

Now, to prove the assertion of Proposition~\ref{prop:impro}, we start considering
the case $0\le \Gamma\le 1$ and assume $\|\ul(t)\|_{L^2}^2=O(t^{-\Gamma})$ as 
$t\to+\infty$.
By Proposition~\ref{prop:enstro}, we have $\|\hl(t)\|_{L^2}^2=o(t^{-1})=O(t^{-\Gamma})$ as $t\to+\infty$, and the assertion of Proposition~\ref{prop:impro}
follows from implication~\eqref{impro1} in this case.
In the case $1< \Gamma\le 2$, from the assumption 
$\|\ul(t)\|_{L^2}^2=O(t^{-\Gamma})$ we have $\|\ul(t)\|_{L^2}^2=O(t^{-1})$ as $t\to+\infty$.
But we have also $\|\hl(t)\|_{L^2}^2=O(t^{-1})$ as $t\to+\infty$, by Proposition~\ref{prop:enstro}. Hence, from implication~\eqref{impro1},
$\|\hl(t)\|_{L^2}^2=O(t^{-2})=O(t^{-\Gamma})$ as $t\to+\infty$.
Now, re-applying implication~\eqref{impro1} gives the conclusion for $1<\Gamma\le 2$. The case $\Gamma>2$ follows by bootstrapping.
\end{proof}

\noindent
\begin{small}
\textbf{Acknowledgments.}
The authors were partially supported by CAPES/COFECUB project N.1040-24 WINPDE.
L. Brandolese was partially supported by ANR-25-CE40-4532.
C. F. Perusato was partially supported by CNPq grant No. 200124/2024-2, CNPq research fellowship No. 310444/2022-5 and by Universit\'e Jean-Monnet.
\end{small} 

\bibliographystyle{myabbrv}
\bibliography{zotero}

\medskip

\begin{description}
\item[Lorenzo Brandolese] 
		Université Lyon 1, ICJ UMR5208, 69622 Villeurbanne, France.\\
Email: \href{mailto:brandolese@math.univ-lyon1.fr}{brandolese@math.univ-lyon1.fr}
\item[Pablo Braz e Silva] Departamento de Matem\'atica. 
		Universidade Federal de Pernambuco, Recife, PE 50740-560. Brazil\\
Email: \href{mailto:pablo.braz@ufpe.br}{pablo.braz@ufpe.br}
\item[Adriana Valentina Busuioc] Université Jean Monnet, CNRS, Centrale Lyon, INSA Lyon, Université  Lyon 1, ICJ UMR5208, 42023 Saint-Etienne, France. \\
Email: \href{mailto:valentina.busuioc@univ-st-etienne.fr}{valentina.busuioc@univ-st-etienne.fr}\
Web page: \url{https://perso.univ-st-etienne.fr/busuvale/}
\item[Dragoş Iftimie] Université Lyon 1, ICJ UMR5208, 69622 Villeurbanne, France.\\
Email: \href{mailto:iftimie@math.univ-lyon1.fr}{iftimie@math.univ-lyon1.fr}\
Web page: \url{http://math.univ-lyon1.fr/~iftimie/}
\item[Cilon F. Perusato] Departamento de Matem\'atica. 
		Universidade Federal de Pernambuco, Recife, PE 50740-560. Brazil \\
Email: \href{mailto:cilon.perusato@ufpe.br}{cilon.perusato@ufpe.br}
\ Web page: \url{http://cilonperusato.com}

\end{description}

\end{document}